\documentclass[12pt]{article}
\usepackage{amsmath,amsthm,amsfonts,amssymb}
\usepackage{sidecap}
\usepackage{float}
\usepackage{extarrows}
\usepackage{booktabs}
\usepackage{verbatim}
\usepackage{hyperref}
\usepackage{mathtools}
\usepackage{esint}
\usepackage[normalem]{ulem}
\usepackage[usenames]{color}
\usepackage{graphicx}

\newtheorem{theorem}{Theorem}[section]
\newtheorem{lemma}[theorem]{Lemma}

\theoremstyle{definition}
\newtheorem{definition}[theorem]{Definition}
\newtheorem{example}[theorem]{Example}

\theoremstyle{remark}
\newtheorem{remark}[theorem]{Remark}

\numberwithin{equation}{section}

\newcommand{\R}{\mathbb{R}}

\newcommand{\RA}{\rightarrow}
\newcommand{\Zn}{\mathbb{Z}^n}
\newcommand{\ttm}{2^{-\frac{m}{2}}}
\newcommand{\XB}{\frac{1}{1+|x|^\beta}}

\newcommand{\sqt}{\sqrt{t}}
\newcommand{\qua}{\frac{1}{4}}

\newcommand{\YBB}{Y_{1,T}^\beta}
\newcommand{\YBBB}{Y_{2, T}^\beta}
\newcommand{\fqt}{\sqrt[4]{t}}
\newcommand{\tfm}{2^{-\frac{m}{4}}}
\newcommand{\falpha}{{\frac{1}{\alpha}}}
\newcommand{\HH}{\mathcal{H}}
\newcommand{\K}{\mathcal{K}}
\usepackage[normalem]{ulem}
\usepackage{cancel}
\newcommand{\stkout}[1]{\ifmmode\text{\sout{\ensuremath{#1}}}\else\sout{#1}\fi}
\let\isout\sout
\renewcommand{\sout}[1]{\ifmmode\text{\isout{\ensuremath{#1}}}\else\isout{#1}\fi}
\newcommand{\pa}{\partial}

\allowdisplaybreaks

\begin{document}

\title{Stability of Self-similar Solutions to Geometric Flows}

\author{Hengrong Du \\ Department of Mathematics\\ Purdue University\\ West Lafayette, IN 47906\\
Email: \href{mailto:du155@purdue.edu}{du155@purdue.edu} \and Nung Kwan Yip\\ Department of Mathematics\\ Purdue University\\ West Lafayette, IN 47906\\
Email: \href{mailto:yip@math.purdue.edu}{yip@math.purdue.edu}}



\date{}


\maketitle
\begin{abstract}
We show that self-similar solutions for the mean curvature flow,  
surface diffusion and  Willmore flow of entire graphs are stable upon  
perturbations of initial data with small Lipschitz norm. 
Roughly speaking, the perturbed solutions are asymptotically self-similar as time tends to infinity.  Our results are built upon the global analytic solutions constructed by Koch and Lamm \cite{KochLamm}, the compactness arguments adapted by Asai and Giga \cite{Giga2014}, and the spatial equi-decay properties on certain weighted function spaces. 
The proof for all of the above flows are achieved
in a unified framework by utilizing the estimates of the linearized 
operator.
\end{abstract}



\section{Introduction}

We analyze in this paper the long-time asymptotics of various geometric flows,
in particular the stability of self-similar solutions. 
From the point of view of calculus of variations, many geometric flows can be 
seen as the negative gradient flows of some geometric functionals with respect 
to certain underlying metric. Heuristically, the gradient descent nature of 
the flows evolves general initial data toward a critical point of 
the corresponding functional. 
These evolutions are often modeled by nonlinear parabolic partial differential 
equations. The long time asymptotics of the solution is one of the key 
questions to be investigated.
For instance, in the celebrated work \cite{Simon1983} of Leon Simon, the asymptotics of a large class of such geometric evolution equations are studied by 
infinite dimensional version of the \L{}ojasiewicz inequalities combined with the Liapunov--Schmidt reduction. 
It is also worth pointing out that in \cite{EllesSampson} Eells and Sampson 
used the long-time limit of heat flows to construct harmonic mappings between Riemannian manifolds under certain curvature assumptions. 

The geometric flows studied in this paper is of curvature 
driven type which arises from energy minimization of the surface area 
functional. This naturally leads to evolutions involving mean curvature 
which is the first variation of the surface area. These motions appear often 
in the modeling of materials science such as phase transitions and grain growth
\cite{AC1979, MullinsTG}. It is also used in describing the bending of 
membranes in red blood cells \cite{helfrich, seifert}. The underlying equations are related to 
mean curvature flows (MCF), surface diffusion (SD) and Willmore flows (WF) 
which are the three equations analyzed in this paper.

One mathematical point to note is that the equations to be 
analyzed include fourth order flows which are much harder to handle than their
second order counterparts due to the lack of maximum or comparison principle.
On the other hand, these flows enjoy certain invariant
property leading to the existence of self-similar solutions. The main goal 
of the current paper is to analyze the stability of these solutions. 
More precisely, under fairly general initial conditions, we will show that
the solutions to these equations converge to some self-similar form.
In order to take advantage of a general unified approach, we restrict ourselves 
to entire graph solutions relying very much on linearized analysis.

One can also interpret this phenomena of self-similarity using the
{\em renormalization group} method as in \cite{BricmontKupiainenLin}. 
The key idea is that after rescaling or zooming out in the spatial 
variable, suppose the initial data converges to a scale invariant function 
which is determined by the behavior of the data at infinity, then
the solution will converge to a
scale invariant solution, or so-called {\em self-similar solution}. 
In other words, the long-time asymptotics are determined by the rescaling limit of the initial data. Hence we expect that if the initial data is 
perturbed without changing the scaling limit, then the corresponding solution will more and more looked like the self-similar solution 
corresponding to the unperturbed scale invariant initial data. 
There is also a huge literature where such a phenomena is proved
for semilinear heat equations - see for example \cite{cazenave, gmiraVeron, kavian, meyer}, just to name a few.
Another technique extensively used in the case of MCF is the 
{\em monotonicity formula}. It 
has been used in this case to characterize the form
of self-similar solutions and the convergence to them
\cite{Huisken, EH}. 
This is also the pre-cursor to the more recent entropy method
to characterize self-similar shrinkers \cite{coldingGenSing}.

In this paper, we will investigate the stability of self-similar solutions 
corresponding to MCF, SD and WF. 
Note that global-in-time existence of classical solutions to these geometric 
flows with general initial data does not hold due to the possibility of finite 
time blow-ups. On the other hand, in the case of graph setting,
it is possible to have long time solutions. For MCF, 
this is comprehensively analyzed in \cite{EH, EHInv}.
In a very interesting paper \cite{KochLamm}, Koch and Lamm has constructed 
a unique global-in-time solution to these geometric flows under 
{\em small Lipschitz norm} assumption on the initial data. 
This is in contrast to those existence results of 
classical solutions making use of maximal regularity property of elliptic 
operators where the initial data are required to be $C^{1,\alpha}$ or 
$C^{2,\alpha}$ (depending on the order of the equation) -- 
see \cite{Simonett1, Simonett2, Simonett3} for examples of such results.
The main technique of \cite{KochLamm}, originated from Koch--Tataru \cite{KochTataru} for incompressible Navier-Stokes equations, is a fixed point argument on 
some scale invariant function spaces. 
Even though it can only handle the case
of graphs, all the above geometric flows in general dimensions can be tackled 
in a unified framework. In addition, the approach does not rely on maximum
principle which only works for second order scalar PDEs. Thus, it is applicable for PDE systems and higher order equations. 

Another relevant work is Asai-Giga \cite{Giga2014} which establishes a
stability result for self-similar solutions to a one dimensional surface 
diffusion with bounded initial data. It uses a compactness argument in some 
H\"{o}lder spaces. An earlier work \cite{duchon}
proves a similar result but it seems the technique is only applicable to the 
one dimensional curve case. From an application point of view, these two works touch upon the 
celebrated model called {\em thermal grooving} first described by 
Mullins \cite{MullinsTG}.
Combining the techniques of \cite{KochLamm} and \cite{Giga2014}, we are 
able to show a local-in-space stability result (Theorem \ref{thm:Main})
and also a global-in-space result (Theorem \ref{thm:GlobalConver}). 
The latter is achieved in the setting of some weighted function spaces.
Qualitatively, we have extended the result of \cite{Giga2014, duchon}
to higher dimensions with unbounded initial data.

This paper is organized as follows. 
In Section \ref{Sec:GeomFlows}, we introduce the geometric flows, the definition of 
self-similar solutions, and the statement of our main results. 
Then we outline the strategy of proof.
Section \ref{sec:localstability} is devoted to the proof of Theorem \ref{thm:Main} which asserts the local-in-space convergence of the perturbed solution. 
Next in Section \ref{sec:equidecay}, we prove our global-in-space 
convergence result (Theorem \ref{thm:GlobalConver}) under a spatial 
decaying assumption on the initial perturbation. We make a remark in Section \ref{sec:polyharm} on the generalization to polyharmonic  flows. The proofs of technical lemmas like Lemma \ref{lemma:keylemma1} and Lemma \ref{lemma:keylemma2} are put in the Appendix.

Before getting into the technical details, we introduce one 
notation to be used throughout this paper. We write 
for any two positive quantities that $A \lesssim B$ if
there is a universal constant $C$ such that $A \leq CB$. The value of
the constant is not relevant in the argument and can change from one
line to the other.

\section{Geometric flows}\label{Sec:GeomFlows}
Let $\Sigma$ be a closed hypersurface in $\R^{n+1}$.
The area functional of $\Sigma$ is given by 
\begin{equation}
  A(\Sigma)=\int_{\Sigma}1d\mu_{g}, 
  \label{eqn:GeomtricFunctional}
\end{equation}
where $g$ is the induced metric from the immersion and $d\mu_g$ is the 
corresponding area element. 
The aim of this paper is to investigate the ($L^2$- and $H^{-1}$-) negative 
gradient flows of \eqref{eqn:GeomtricFunctional}. 
More precisely, we consider a time dependent hypersurface 
$\Sigma_t$ given by immersions $f:\Sigma\times \R_+\rightarrow \R^{n+1}$ 
which evolves according to
\begin{description}
\item[Mean curvature flow(MCF)] 
  \begin{equation}
    \partial_t f =\mathcal{H}:=-\nabla_{L^2} A,
    \label{eqn:MCF}
  \end{equation}
\item[Surface diffusion(SD)] 
  \begin{equation}
    \partial_t f = -\Delta_g \mathcal{H}=:-\nabla_{H^{-1}} A,
    \label{eqn:SDF}
  \end{equation}
\end{description}
where $\mathcal{H}$ represents the mean curvature vector and $\Delta_g$ is the 
Laplace--Beltrami operator with respect to the induced metric $g$. 
Note that MCF and SD can be 
recasted as the negative gradient flows to $A$ with respect to the $L^2$ and 
$H^{-1}$-metric. See \cite{taylor-cahn-1, taylor-cahn-2} 
for more details about the derivation.

We will also consider the following Willmore functional for two dimensional
surfaces ($n=2$) in $\mathbb{R}^3$:
 \begin{equation}
W(f)=\frac{1}{4}\int_{\Sigma}|\mathcal{H}|^2 d\mu_g.
\label{willmore}
\end{equation}
The negative $L^2$-gradient flow of \eqref{willmore} is then given as follows:
    \begin{description}
	\item[Willmore flow(WF)]
	\begin{equation}
	f_t^{\perp}=-\Delta_g \HH-\frac{1}{2}\HH^3+2\HH\K=:-\nabla_{L^2}W,
	\label{eqn:WF}
	\end{equation}
\end{description}
where $\K$ is the Gauss curvature of $\Sigma$. We refer the reader to 
\cite{kuwert-schatzle} for detail of the derivation.

 As mentioned earlier, in this paper, we consider the case that $\Sigma_t$ given by an entire 
graph, i.e. there exists a function $u:\R^n\times\R_+\rightarrow 
\R$ such that $\Sigma_t=\left\{ (x, u(x,t))|x\in \R^n, t\in\R_+ \right\}$. 
For concreteness, we write down the graph equations for 
\eqref{eqn:MCF}, \eqref{eqn:SDF} and \eqref{eqn:WF}:
\begin{eqnarray}
\text{\bf MCF}\,\,\,\,\,\,&&
\frac{\partial u}{\partial t} = \sqrt{1 + |\nabla u|^2}\text{div}\Big(
\frac{\nabla u}{\sqrt{1 + |\nabla u|^2}}
\Big),\\
\text{\bf SD}\,\,\,\,\,\,&&
\frac{\partial u}{\partial t} = 
-\text{div}\left[
\sqrt{1 + |\nabla u|^2}
\left(
I - \frac{\nabla u \otimes \nabla u}{1 + |\nabla u|^2}
\right)
\nabla\text{div}\left(
\frac{\nabla u}{\sqrt{1 + |\nabla u|^2}}
\right)
\right],\\
\text{\bf WF}\,\,\,\,\,\,&&
\frac{\partial u}{\partial t} = -w\text{div}\left[
\frac{1}{w}\left(
\left(
I - \frac{\nabla u \otimes\nabla u}{w^2}
\right)\nabla(w\mathcal{H})
-\frac{1}{2}\mathcal{H}^2\nabla u
\right)
\right].
\end{eqnarray}
In the above, we have used the following notations and representation
\[
w = \sqrt{1+|\nabla u|^2}
\,\,\,\,\,\,\text{and}\,\,\,\,\,\,
\mathcal{H} = \text{div}\left(\frac{\nabla u}{w}\right).
\] 

To simplify the above equations,
we borrow the contraction operator $\star$ from \cite{KochLamm} for all possible contractions between derivatives of $u$, for example, we use 
$\nabla^2 u\star\nabla u\star \nabla u$ to indicate any expression of the
form $\nabla_{ij}u\nabla_k u\nabla_l u$ with $1 \leq i,j,k,l\leq n$.
They are all treated equally in terms of analysis.
Moreover, we use $P_k(\nabla u)$ to denote some $k$-th power contraction 
of $\nabla u$, i.e., 
$$P_k(\nabla u)
=\underbrace{\nabla u\star\cdots \star\nabla u}_{k\text{-times}}
= \Pi_{j=1}^k\nabla_{i_j}u,\,\,\,\text{for some $1\leq i_j\leq n$}.
$$
As derived in \cite{KochLamm}, we can rewrite the equation 
\eqref{eqn:MCF}, \eqref{eqn:SDF} and \eqref{eqn:WF} using the above convention as follows:
\begin{description}
  \item[MCF] 
    \begin{equation}
      \partial_t u-\Delta u=w^{-2}\nabla^2 u\star P_2(\nabla u),
      \label{eqn:GrapMCF}
    \end{equation}
  \item[SD]
    \begin{equation}
      \partial_t u+\Delta^2 u=\nabla_if_1^i[u]+\nabla_{ij}f_2^{ij}[u],
      \label{eqn:GrapSDF}
    \end{equation}
  \item[WF] 
    \begin{equation}
      \partial_t u+\Delta^2 u=f_0[u]+\nabla_i f_1^i[u]+\nabla_{ij}f_2^{ij}[u],
      \label{eqn:GrapWF}
    \end{equation}
\end{description}
where  
\begin{eqnarray}
f_0[u]&=&\nabla^2 u\star\nabla^2 u\star \nabla^2 u\star \sum_{k=1}^{4}w^{-2k}P_{2k-2}(\nabla u),\label{f0}\\
f_1[u]&=&\nabla^2 u\star \nabla^2 u\star \sum_{k=1}^{4}w^{-2k}P_{2k-1}(\nabla u),\label{f1}\\
f_2[u]&=&\nabla^2 u\star\sum_{k=1}^{2}w^{-2k}P_{2k}(\nabla u).
\label{f2}
\end{eqnarray}
Under the assumption that $|\nabla u|\lesssim 1$, 
the following crude bounds for the nonlinear terms play crucial roles 
in our analysis:
\begin{align}
|f_0[u]|\lesssim |\nabla^2u|^3, \quad
|f_1[u]|\lesssim |\nabla^2 u|^2,
\,\,\, \text{ and }\,\,\,
|f_2[u]|\lesssim |\nabla^2 u|.
\label{eqn:WFnonlinear}
\end{align}

 Abstractly, we can write \eqref{eqn:GrapMCF}, \eqref{eqn:GrapSDF}, and 
\eqref{eqn:GrapWF} in the following form
\begin{equation}
  \left\{
    \begin{array}{l}
      \partial_t u+A u=N[u], \quad (x, t)\in \R^n\times(0, \infty), \\
      u(x, 0)=u_0(x), \quad x\in \R^n,
    \end{array}
    \right.
    \label{eqn:AbstractPDE}
  \end{equation}
  where $A=-\Delta$ or $\Delta^2$, and $N[u]$ is the nonlinear term 
in the right hand sides of \eqref{eqn:GrapMCF}, \eqref{eqn:GrapSDF}, or 
\eqref{eqn:GrapWF}.
   We say $u(x,t)$ is a \emph{mild solution} to \eqref{eqn:AbstractPDE} 
if it satisfies the following integral equation
  \begin{equation}
  u(x,t)=e^{-A t}u_0(x)+\int_{0}^{t}e^{-(t-s)A}N[u](x, s)ds, \quad(x,t)\in\R^n\times(0,\infty).
  \label{}
  \end{equation}
  where $e^{-A t}$ is the semigroup generated by $-A$.
    If the Lipschitz norm of $u_0$ is small, the global well-posedness of mild solution to \eqref{eqn:AbstractPDE} is obtained by Koch--Lamm \cite{KochLamm}. 

One of the most important features of these equations is their 
{\em scale invariant} property. More precisely, for any positive constant $\lambda$,
if we define $\Sigma_\lambda := \lambda^{-1}\Sigma$, then 
$$
\mathcal{H}_{\Sigma_\lambda} = \lambda\mathcal{H}_\Sigma,\,\,\,
\mathcal{K}_{\Sigma_\lambda} = \lambda^2\mathcal{K}_\Sigma,\,\,\,
\text{and}\,\,\,
\triangle_{\Sigma_\lambda} = \lambda^2\triangle_\Sigma.
$$
In terms of equation, let $u$ be a mild solution to 
\eqref{eqn:AbstractPDE}. If we similarly define
$u_{\lambda}(x, t):=\lambda^{-1}u(\lambda x, \lambda^{\alpha}t)$, 
where $\alpha=2$ if $A=-\Delta$ and $\alpha=4$ if $A=\Delta^2$.
Then $u_\lambda$ solves the same PDE but with rescaled initial data, i.e., 
  \begin{equation}
    \left\{
      \begin{array}{l}
	\partial_t u_\lambda+Au_\lambda=N[u_\lambda], \quad (x, t)\in \R^n\times (0, \infty), \\
	u_\lambda(x, 0)=\lambda^{-1}u_0(\lambda x), \quad x\in \R^n.
      \end{array}
      \right.
      \label{}
    \end{equation}
Note that with $y=\lambda x$, then $\nabla_x u_\lambda = \nabla_y u$,
$\nabla^2_x u_\lambda = \lambda \nabla^2_y u$ and so forth. 
The powers of $\nabla^2 u$ 
in the nonlinear terms $f_i$ are such that 
$f_i(u_\lambda) = \lambda^{3-i}f_i(u)$ for $i=0,1,2$. Hence
$\nabla^i f_i(u_\lambda) = \lambda ^3\nabla^i f_i(u)$.
They indeed give the corresponding scale invariance with $\alpha=4$
for SD and WF. For MCF, we only have the term $f_2(u) \sim \nabla^2u$, 
corresponding to $\alpha=2$. 

The above naturally leads to the notion of self-similar 
solutions $v$ which satisfy $v_\lambda(x,t) = v(x,t)$. Setting $t=0$,
then the initial data necessarily has the property that
$v(x,0) = \lambda^{-1}v(\lambda x,0)$. 
Conversely, let $v$ be 
the solution of \eqref{eqn:AbstractPDE} with self-similar initial data 
$v_0(x)=|x|\psi\left(\frac{x}{|x|}\right)$ for some function 
$\psi: \mathbb{S}^{n-1}\longrightarrow\R$ so that $v_0$ is indeed
self-similar, $v_0(x)=\lambda^{-1}v_0(\lambda x)$. 
Since $v_\lambda$ solves the same equation and initial data, by 
the uniqueness of solution, it holds that $v_\lambda(x, t)=v(x,t)$. 
Upon introducing $\Psi(y)=v(y,1)$, we then have 
    \begin{equation}
      v(x,t)=v_{t^{-\frac{1}{\alpha}}}(x, t)=t^{\frac{1}{\alpha}} v(xt^{-\frac{1}{\alpha}}, 1)=:t^\frac{1}{\alpha} \Psi(xt^{-\frac{1}{\alpha}}).
      \label{eqn:SelfProfile}
    \end{equation}
The function $\Psi$ is called a \emph{self-similar profile}
and it satisfies the following equation:
$$
A\Psi(y) + \frac{1}{\alpha}\Psi(y) - \frac{1}{\alpha}y\cdot\nabla\Psi(y)
= N(\Psi(y)).
$$

The main objective of this paper is to study the stability of 
self-similar solutions under bounded (and small) perturbation of 
self-similar initial data. Our main results are given as follows:
    \begin{theorem}
      There exists an $\varepsilon>0$ such that if $u(x,t)$ is a global mild solution to \eqref{eqn:AbstractPDE} with perturbed self-similar initial data of 
$u_0(x)=v_0(x)+p(x)$ such that
$\|p\|_{L^\infty(\R^n)}<\infty$ and 
$\|\nabla v_0\|_{L^\infty(\R^n)}$, 
$\|\nabla p\|_{L^\infty(\R^n)} <\varepsilon$, 
then for any compact subset $K$ of $\R^n$, it holds that
      \begin{equation}
      	\lim_{t\rightarrow \infty}\left\|t^{-\frac{1}{\alpha}}u(t^{\frac{1}{\alpha}} x, t)-\Psi(x)\right\|_{C^k(K)}=0, \quad \forall k\in\mathbb{N}^+. \label{eqn:asymptotic}
      \end{equation}
      \label{thm:Main}
    \end{theorem}
The next example demonstrates the validity of 
Theorem \ref{thm:Main}.
\begin{example}
        Consider a shifting perturbation on initial self-similar data 
$v_0(x)$ by $a\in \R^n$, i.e., $u_0(x)=v_0(x-a)$. In this case, $p(x)=u_0(x)-v_0(x)=v_0(x-a)-v_0(x)$, which satisfies the condition of Theorem \eqref{thm:Main}. In fact, {{      \begin{align*}
                &\|\nabla v_0\|_{L^\infty(\R^n)}+\|\nabla p\|_{L^\infty(\R^n)}\\
                \leq &\|\nabla v_0\|_{L^\infty(\R^n)}+\|\nabla(v_0(\cdot-a)-v_0(\cdot))\|_{L^\infty(\R^n)}\\
                \le &3\|\nabla v_0\|_{\infty(\R^n)}< 3\varepsilon,
        \end{align*}}}
         and         
        \begin{align*}
        \|p(x)\|_{L^\infty(\R^n)}&=\|v_0(x-a)-v_0(x)\|_{L^\infty(\R^n)}\\
        &\leq{{\|\nabla v_0\|_{L^\infty(\R^n)}}}|a|<\infty.
        \end{align*}
        From {{the}} uniqueness of the mild solution to \eqref{eqn:AbstractPDE}, we have $$u(x,t)=v(x-a, t)=t^\falpha \Psi\left( (x-a)t^{-\falpha} \right),$$then we can show
        \begin{align*}
          \lim_{t\RA\infty}\left\|t^{-\falpha}u(t^{\falpha}x, t)-\Psi(x)\right\|_{C^k(K)}    
          &=\lim_{t\RA\infty}\left\|\Psi\left( x-at^{-\falpha} \right)-\Psi(x)\right\|_{C^k(K)}\\
          &=0, \quad \forall k\in \mathbb{N}^+,
        \end{align*}
        since $\Psi$ is smooth. 
\end{example}

We also have the following result on the global convergence 
under perturbation with spatial decay.
\begin{theorem}[Global stability with spatial decay]
There exists an $\varepsilon>0$ such that if $u(x, t)$ is a global mild solution to \eqref{eqn:AbstractPDE} with perturbed self-similar initial data of 
$u_0(x)=v_0(x)+p(x)$ such that
$\|p\|_{L^\infty(\R^n)}<\infty$ and
$\|\nabla v_0\|_{L^\infty(\R^n)} +
\|(1+|x|^\beta)\nabla p\|_{L^\infty(\R^n)}
<\varepsilon$, 
then we have
  \begin{equation}
    \lim_{t\RA\infty}\left\|t^{-\falpha}u(t^{\falpha}x, t)-\Psi(x)\right\|_{C^1(\R^n)}=0.
    \label{eqn:GlobalConver}
  \end{equation}
  \label{thm:GlobalConver}
\end{theorem}
\begin{remark}
	It seems possible to also prove higher order global-in-space 
convergence results. The main technical step is to generalize Lemma \ref{lemma:keylemma1} and 
Lemma \ref{lemma:keylemma2} to higher order estimates. The work \cite{KochLamm}
uses analytic Banach fixed point theorem to obtain higher order regularity.
For the reason of conciseness and space, we omit this step in this paper.
\end{remark}

The following result (global well-posedness for initial data with
small Lipschitz norm) for \eqref{eqn:AbstractPDE} and the technique to prove 
it provide a starting point for our investigation. (The definition of the
function space $X_\infty$ will be given in Section \ref{sec:equidecay}.)
    \begin{theorem}[Koch--Lamm \cite{KochLamm}, Theorem 3.1 \& 5.1]
      There exists $\varepsilon>0$, $C>0$ such that for every $u_0$ with
$\|\nabla u_0\|_\infty<\varepsilon$ there exists an analytic solution $u\in X_\infty$ of \eqref{eqn:AbstractPDE} with $u(\cdot , 0)=u_0$ which satisfies $\left\|u\right\|_{X_\infty}\le C \|\nabla u_0\|_{L^\infty(\R^n)}$. The solution is unique in the ball $B^{X_\infty}_{C\varepsilon}(0):=\{u\in X_\infty|\|u\|_{X_\infty}\le C\varepsilon\}.$ Moreover, there exist $R>0$, $c>0$ such that for every $k\in \mathbb{N}_0$ and multi-index
      $\gamma\in \mathbb{N}^n_0$, we have the estimate
      \begin{equation}
        \sup_{x\in \R^n}\sup_{t>0}\left|
(t^{\frac{1}{\alpha}}\nabla)^\gamma 
(t\partial_t)^k \nabla u(x, t)\right|
\le c\|\nabla u_0\|_{L^\infty(\R^n)} R^{|\gamma|+k}(|\gamma|+k)!.
        \label{eqn:KLregularity}
      \end{equation}
Furthermore, $u$ depends analytically on $u_0$.
      \label{thm:KochLammRegu}
    \end{theorem}

Note that even though the estimate resembles those coming from linear 
parabolic equations and is consistent with the scale invariant property, 
it is highly nontrivial to establish for nonlinear equations. The fact 
that the estimates are expressed in terms of the
Lipschitz norm of the initial data is particularly useful as self-similar 
initial data is necessarily only Lipschitz. Furthermore,
note the following gradient bound for the solution ($\gamma=0, k=0$):
\begin{equation}
\|\nabla u(t)\|_{L^\infty(\mathbb{R}^n)}
\leq C\|\nabla u_0\|_{L^\infty(\mathbb{R}^n)}
\end{equation}
implies that the smallness of the Lipschitz norm is preserved in 
time. This fact is crucial if we want to work in the graph setting
because for surface diffusion, it has been shown by 
\cite{EMP} that in general the graph property might not be preserved.

For the rest of this section, we outline the strategy 
of the proof of Theorem \ref{thm:Main}. Such an approach is also described
in \cite[Chapter 1]{GigaBk} by M.-H. Giga, Y. Giga and J. Saal.
First, note that upon setting $\lambda=t^\falpha$, then
$u_\lambda(x,1) =u_{t^\falpha}(x, 1) = t^{-\falpha}u(xt^{\falpha}, t)$.
Hence
\eqref{eqn:asymptotic} is equivalent to
 \begin{align}\label{eqn:RescalingDifference}
  \lim_{\lambda\rightarrow\infty}\left\|u_\lambda(x, 1)-v_\lambda(x,1)\right\|_{C^k(K)}
   =\lim_{\lambda\rightarrow \infty}\left\|u_\lambda(x, 1)-v(x, 1)\right\|_{C^k(K)}
   =0, \quad \forall k\in\mathbb{N}^+.
 \end{align}
Thus all we need is to estimate at time $t=1$ the difference between the two 
solutions $u_\lambda$ and $v_\lambda \equiv v$.
Now let $\Phi_\lambda:=u_\lambda-v$. Then it satisfies
\begin{equation}
  \Phi_\lambda(x, t)=e^{-At}p_\lambda(x)+\int_{0}^{t}e^{-(t-s)A}(N[v+\Phi_\lambda]-N[v])(x, s)ds, 
  \label{eqn:DifferenceEquation}
\end{equation}
where we have used the fact that the difference between the
initial data is given by 
$u_\lambda(x,0) - v(x,0) 
= \frac{1}{\lambda}p(\lambda x) 
:= p_\lambda(x)$. 

Next, the following estimate from Theorem \ref{thm:KochLammRegu} is 
applicable to both $u_\lambda$ and $v$: 
 \begin{equation}
   |\nabla^{{\gamma}}\partial_t^k \nabla u(x,t)|\leq C t^{-\left( \frac{|{{\gamma}}|}{\alpha}+k \right)}\left\|\nabla u_0\right\|_{L^\infty(\R^n)}.
   \label{eqn:Regularity}
 \end{equation}
Putting \eqref{eqn:Regularity} and \eqref{eqn:DifferenceEquation} together, we can apply Arzela--Ascoli compactness theorem to show that
there is a subsequence
$\{\Phi_{\lambda_k}\}, \lambda_k\rightarrow \infty$ 
and $\Phi_\infty\in C^\infty(\R^n\times(0,1])$ such that 
the following statements hold.
\begin{enumerate}
		\item (Convergence) For any compact subset $K$ of $\R^n$, 
	\begin{equation}
	\lim_{\lambda_{k}\rightarrow\infty} \left\|\Phi_{\lambda_{k}}(x, 1)-\Phi_\infty(x, 1)\right\|_{C^k(K)}=0, \quad \forall k\in\mathbb{N}.
	\label{eqn:convergence1}
	\end{equation}
	\item (Regularity) For any $t\in (0, 1]$, 
	\begin{equation}
	\|\nabla^{{\gamma}}\partial_t^k \nabla\Phi_\infty(\cdot,t)\|_{L^\infty(\R^n)}\leq C t^{-\left( \frac{|{{\gamma}}|}{\alpha}+k \right)}({{\|\nabla v_0\|_{L^\infty(\R^n)}+\|\nabla p\|_{L^\infty(\R^n)}}}).
	\label{eqn:PhiRegularity}
	\end{equation} 
	\item (Integral equation) $\Phi_{\infty}(x,t)$ solves the following integral equation:
	\begin{equation}
	\Phi_\infty(x, t)=\int_{0}^{t}e^{-(t-s)A}(N[v+\Phi_\infty]-{{N[v]}})(x, s)ds, \quad (x,t)\in \R^n\times(0,1]. 
	\label{eqn:fixpointequation}
	\end{equation}
\end{enumerate}
As the last step, we conclude the proof of \eqref{eqn:RescalingDifference} by 
showing that every solution $\Phi_\infty$ of \eqref{eqn:fixpointequation} 
satisfying the property $\|\nabla \Phi_\infty\|_{L^\infty(\mathbb{R}^n)}\ll 1$ 
and the regularity estimate \eqref{eqn:PhiRegularity}
must be equal to $0$.

We would like to emphasize that the above strategy is very simple and 
robust. See again \cite{GigaBk} for a general exposition of this strategy.
Despite the fact that the results are restricted to the graph
setting, it is applicable to all the geometric evolutions under 
consideration here. Another advantage is that maximum or comparison principle 
is not used in the current approach. See for example the results for MCF
\cite{stavrou, clutterbuckSchnurer, cesaroniNovaga} that do rely on 
such a principle.

As a last remark before presenting the proof, note that WF has one more 
term, $f_0[u]$, than SD. 
Thus in the current work, we will only consider MCF and WF for simplicity. 

\section{Stability Result - Local Version}\label{sec:localstability}
In this section, we will prove Theorem \ref{thm:Main}. 
As outlined above, we
will first establish uniform estimates and compactness of $\Phi_\lambda$.
In all of the following result, we are working in the regime of {\em small 
Lipschitz norm}. More precisely, there exist an $\epsilon \ll 1$ such that
$
\|\nabla u_0\|_{L^\infty(\mathbb{R}^n)},
\|\nabla v_0\|_{L^\infty(\mathbb{R}^n)}
\ll 1$.

\subsection{MCF} 
In this case, we have $A=-\Delta$, $\alpha=2$. Thus equation 
\eqref{eqn:DifferenceEquation} for $\Phi_\lambda$ becomes
\begin{equation}
  \Phi_\lambda(x, t)=e^{{{t\Delta}} }p_\lambda(x)+\int_{0}^{t}e^{(t-s)\Delta}\left( N[u_\lambda]-N[v] \right)(x, s)ds.
  \label{eqn:MCFdifference}
\end{equation}
The nonlinear term $N[u]$ can be estimated as:
\begin{equation}\label{nonlin-est-MCF}
|N[u]| 
 =(1+|\nabla u|^2)^{-1}\nabla u\star \nabla u\star \nabla^2 u
\lesssim |\nabla u|^2|\nabla^2 u| \lesssim |\nabla^2 u|.
\end{equation}
We also recall the heat kernel and its associated semigroup:
\begin{equation}\label{heat.kernel}
h(x,t):=\frac{1}{(4\pi t)^{\frac{n}{2}}}\exp\left(-\frac{|x|^2}{4t}\right),
\,\,\,\,\,\,\text{and}\,\,\,\,\,\,e^{\Delta t}f(x):=\int_{\R^n}h(x-y, t)f(y)dy.
\end{equation}

\subsubsection{Uniform estimates and Compactness for $\Phi_\lambda$}
We first note several useful facts. 
By the $L^1$-bound of the heat kernel, we get
 \begin{equation}
\sup_{\lambda>1}\sup_{t\geq 0}\left\|e^{\Delta t}p_\lambda(x)\right\|_{L^\infty(\R^n)}\leq \sup_{\lambda>1}\left\|p_\lambda(x)\right\|_{L^\infty(\R^n)}<\infty.\label{eqn:plambdabd}
\end{equation}
Furthermore, the Lipschitz norm is invariant under the rescaling:
\begin{equation}
\left\|\nabla p_\lambda\right\|_{L^\infty(\R^n)}
=\left\|\nabla p\right\|_{L^\infty(\R^n)}.
\end{equation}
From the regularity estimate \eqref{eqn:Regularity}, we have
 \begin{align} 
   \nonumber
   \sup_{\lambda>1}\left\|\partial_t^k\nabla^{{{\gamma}}}\nabla u_\lambda(\cdot , t)\right\|_{L^\infty(\R^n)}
   \lesssim& t^{-\frac{|{{\gamma}}|}{2}-k}{{\sup_{\lambda>1}}}\left\|\nabla(v_0+p_\lambda)\right\|_{L^\infty(\R^n)}\\
   \lesssim& t^{-\frac{|{{\gamma}}|}{2}-k}({{\|\nabla v_0\|_{L^\infty(\R^n)}+\|\nabla p\|_{L^\infty(\R^n)}}})  
\label{eqn:uniformregularity1}
 \end{align}
and similarly for $v_\lambda = v$,
\begin{equation}\label{v.est}
\sup_{\lambda>1}\left\|\partial_t^k\nabla^{{{\gamma}}}\nabla v_\lambda(\cdot , t)\right\|_{L^\infty(\R^n)}\\
   \lesssim t^{-\frac{|{{\gamma}}|}{2}-k}\|\nabla v_0\|_{L^\infty(\R^n)}.
\end{equation}

Now we estimate
 \begin{align} 	\label{eqn:MCFunifb1}
 	&\sup_{\lambda>1}\|\Phi_{\lambda}(\cdot, t)\|_{L^\infty(\R^n)}\\
 \leq& \sup_{\lambda>1}\|e^{\Delta t}p_\lambda(\cdot)\|_{L^\infty (\R^n)}+\sup_{\lambda>1}\left\|\int_{0}^t e^{-(t-s)\Delta}(N[u_\lambda]-N[v])(\cdot,s)ds\right\|_{L^\infty(\R^n)}\nonumber\\
 	\lesssim &\sup_{\lambda>1}\|e^{\Delta t}p_\lambda(\cdot)\|_{L^\infty (\R^n)}+\sup_{\lambda>1}\left\|\int_0^t \int_{\R^n}|h(\cdot -y,t-s)|(|N[u_\lambda]|+|N[v]|)(y,s)dyds\right\|_{L^\infty(\R^n)}\nonumber\\
 \lesssim& \sup_{\lambda >1}\|p_\lambda\|_{L^\infty(\R^n)}+({{\|\nabla v_0\|_{L^\infty(\R^n)}+\|\nabla p\|_{L^\infty(\R^n)}}})\left\|\int_0^t\int_{\R^n} h(\cdot-y, t-s) s^{-\frac{1}{2}}dyds\right\|_{L^\infty(\R^n)}\nonumber\\	
\lesssim& 
\sup_{\lambda >1}\|p_\lambda\|_{L^\infty(\R^n)}+
({{\|\nabla v_0\|_{L^\infty(\R^n)}+\|\nabla p\|_{L^\infty(\R^n)}}})\nonumber\\
<&\infty.\nonumber
 \end{align}
In the above, we have used the estimate 
\begin{equation}
\|N[u_\lambda(\cdot, s)]\|_{L^\infty(\R^n)}
\lesssim
\|\nabla u_\lambda(\cdot, s)\|^2_{L^\infty(\R^n)}\|\nabla^2 u_\lambda(\cdot, s)\|_{L^\infty(\R^n)}
\lesssim
s^{-\frac{1}{2}}
\label{nonlin-est-MCF-t}
\end{equation}

 By the higher order regularity estimate 
 \eqref{eqn:uniformregularity1} and \eqref{v.est}, 
we have for any $k\in \mathbb{N}, \gamma\in \mathbb{N}^n$,
 \begin{align}\label{eqn:MCFunifb2}
\sup_{\lambda>1}\|\nabla^{{\gamma}}\partial_t^k \nabla\Phi_\lambda(\cdot,t)\|_{L^\infty(\R^n)}
\lesssim&\sup_{\lambda>1}\|\nabla^{\gamma}\pa_t^k\nabla u_\lambda(\cdot,t)\|_{L^\infty(\R^n)}+\|\nabla^\gamma \pa_t^k \nabla v(\cdot, t)\|_{L^\infty(\R^n)}\nonumber\\
\lesssim &({{\|\nabla v_0\|_{L^\infty(\R^n)}+\|\nabla p\|_{L^\infty(\R^n)}}}) t^{-\frac{|{{\gamma}}|}{2}-k}.\nonumber
 \end{align} 

With the above uniform estimates for $\Phi_\lambda$, we can apply
the Arzela--Ascoli theorem 
to extract a subsequence $\left\{ \Phi_{\lambda_k} \right\}$ and 
$\Phi_\infty(x, t)\in C^\infty({\color{black}\R^n\times(0,1]})$ 
such that for any $\delta>0$, compact subset $K\subset\R^n$,
and $k\in \mathbb{N}$, we have
\begin{equation}
	\lim_{\lambda_k\rightarrow\infty}\sup_{\delta\leq t\leq 1}
	\|\Phi_{\lambda_k}-\Phi_{\infty}\|_{C^k(K)}=0.
\label{eqn:convergence2}
\end{equation}
Then \eqref{eqn:convergence1} and \eqref{eqn:PhiRegularity} follow.

\subsubsection{Equation for $\Phi_\infty$}
Here we verify \eqref{eqn:fixpointequation} by passing the limit $\lambda_k\RA\infty$ in \eqref{eqn:MCFdifference}. First note that
\begin{equation}
\lim_{\lambda\rightarrow \infty}\sup_{t\geq 0}\|e^{\Delta t}p_\lambda(\cdot)\|_{L^\infty(\R^n)}\leq \lim_{\lambda\rightarrow \infty}\|p_\lambda\|_{L^\infty(\R^n)}={\color{black}\lim_{\lambda\to\infty}\frac{1}{\lambda}\|p\|_{L^\infty(\R^n)}}=0.
\label{eqn:plambdaconvMCF}
\end{equation}
Second, from \eqref{eqn:convergence2}, we know that 
for any $\delta>0$ and any compact subset $K\subset\R^n$, 
\begin{equation}
\lim_{\lambda_k\rightarrow\infty}\sup_{\delta\leq t\leq 1}\|N[v+\Phi_{\lambda_k}]-N[v+\Phi_{\infty}]\|_{C^k(K)}=0.
\label{eqn:nonlinearconver1}
\end{equation}
Now note that
\begin{align*}
& \left|\int_{0}^{t}e^{(t-s)\Delta}(N[\Phi_\lambda+v]-N[\Phi_\infty+v])(x,s)ds\right|\\
\leq & \int_{0}^{t}\int_{\mathbb{R}^n} h(t-s, x-y)
\Big[|N[\Phi_\lambda+v]|+|N[\Phi_\infty+v]|\Big](y,s)\,dyds.
\end{align*}
By the formula of the heat kernel \eqref{heat.kernel} and 
the estimate for the nonlinear term \eqref{nonlin-est-MCF-t},
the integrand can be estimated as:
$$
h(t-s, x-y)
\Big[|N[\Phi_\lambda+v]|+|N[\Phi_\infty+v]|\Big](y,s)
\lesssim
(t-s)^{-\frac{n}{2}}
\exp\left( -\frac{|x-y|^2}{4(t-s)}\right)s^{-\frac{1}{2}}
$$
which is integrable:
$$
\int_0^t\int_{\mathbb{R}^n}
(t-s)^{-\frac{n}{2}}
\exp\left( -\frac{|x-y|^2}{4(t-s)}\right)s^{-\frac{1}{2}}\,dy\,ds
\lesssim t^{\frac{1}{2}}.
$$
Hence \eqref{eqn:fixpointequation} follows by the Lebesgue Dominated 
Convergence Theorem.

\subsection{WF} 
In this case, we have $A=\Delta^2, \alpha=4$, and 
 $$N[u]=f_0[u]+\nabla_i f_1^i[u]+\nabla^2_{ij} f_2^{ij}[u].$$
First we introduce the heat kernel of biharmonic operator $b(x,t)$:
 $$b(x,t)=t^{-\frac{n}{4}}g\left( \frac{x}{t^{\frac{1}{4}}} \right),
 \quad\text{where}\quad
 g(\xi)=(2\pi)^{-\frac{n}{2}}\int_{\R^n}e^{i\xi\cdot k-|k|^4} dk, \quad \xi\in \R^n.$$ 
 Furthermore, 
it satisfies the following decaying estimates 
(see \cite[Chapter 9, Theorem 7]{friedman2008partial}, \cite{KochLamm}) which play a very important role in this 
paper:
 \begin{eqnarray}
   |b(x,t)| 
& \lesssim & t^{-\frac{n}{4}}\exp\left( -C\frac{|x|^{\frac{4}{3}}}{t^{\frac{1}{3}}} \right), 
   \label{eqn:KerDecay1}\\
   |\nabla^k b(x,t)|
& \lesssim & t^{-\frac{n+k}{4}}{{\exp\left(-C_k\frac{|x|^{\frac{4}{3}}}{t^{\frac{1}{3}}}\right)}},\quad \forall k\geq 1, 
   \label{eqn:KerDecay2}
 \end{eqnarray}
 The integral equation for mild solutions $u(x, t)$ to \eqref{eqn:GrapWF} now reads
 \begin{align}   \label{eqn:WFmild}
   u(x,t)&=\int_{\R^n}b(x-y,t)u_0(y)dy+\int_{0}^{t}\int_{\R^n}b(x-y,t-s)f_0[u](y,s)dyds\\
   &\qquad-\int_{0}^{t}\int_{\R^n}\nabla_i b(x-y,t-s)f_1^i[u](y,s)dyds\nonumber\\
   &\qquad+\int_{0}^{{{t}}}\int_{\R^n}\nabla^2_{ij}b(x-y,t-s)f_2^{ij}[u](y,s)dyds.\nonumber
 \end{align}
Given the uniform bound for $\|\nabla u\|_{L^\infty(\R^n)} \lesssim 1$, 
we note here the estimates for the nonlinear structures:
\begin{equation}
|f_0[u]| \lesssim |\nabla^2 u|^3 \lesssim t^{-\frac{3}{4}},\,\,\,\,\,\,
|f_1[u]| \lesssim |\nabla^2 u|^2 \lesssim t^{-\frac{2}{4}},\,\,\,\,\,\,
|f_2[u]| \lesssim |\nabla^2 u| \lesssim t^{-\frac{1}{4}}.
\label{nonlin.struct.est.WF}
\end{equation}
Note also that in order to take advantage of the kernel decay, we perform 
integration by parts to eliminate the derivatives on $f_1$ and $f_2$. With 
this, we use the following $L^1$-bound for $b$,
\begin{equation}
\|\nabla^k b(\cdot, t)\|_{L^1(\mathbb{R}^n)} \lesssim t^{-\frac{k}{4}}
\quad\text{for $k=0, 1, 2$}.
   \label{eqn:KerDecay3}
\end{equation}

\subsubsection{Uniform estimates and convergence for $\Phi_\lambda$}
Using the estimates for $b$, we first establish $L^\infty$ bound for 
$\Phi_\lambda$. For $e^{-\Delta^2 t}p_\lambda$, we have
 \begin{align}
   \sup_{\lambda>1}\sup_{t\geq 0}\left\|e^{-\Delta^2 t}p_\lambda\right\|_{L^\infty}&=\sup_{\lambda>1}\sup_{t\geq 0}\left\|\int_{\R^n}b(\cdot-y,t)p_\lambda(y)dy\right\|_{L^\infty}\nonumber\\ &\leq\sup_{\lambda>1}\sup_{t\geq 0}\left\|p_\lambda\right\|_{L^\infty}\left\|b(\cdot , t)\right\|_{L^1}\nonumber\\
   &\lesssim \sup_{\lambda >1}\left\|p_\lambda\right\|_{L^\infty}<\infty.
\label{WFp}
 \end{align}
 From the regularity estimates \eqref{eqn:Regularity} we have
 \begin{align}   
   \nonumber
   \sup_{\lambda>1}\left\|\partial_t^k\nabla^{{\gamma}}\nabla u_\lambda(\cdot , t)\right\|_{L^\infty}
\lesssim&  t^{-\frac{|{{\gamma}}|}{4}-k}{{\sup_{\lambda>1}}}\left\|\nabla(v_0+p_\lambda)\right\|_{L^\infty}\\
   \lesssim& {{t^{-\frac{|\gamma|}{4}-k}\left(\|\nabla v_0\|_{L^\infty(\R^n)}+\sup_{\lambda>1}\|\nabla p_\lambda\|_{L^\infty(\R^n)}\right)}}\nonumber\\
   \lesssim& t^{-\frac{|{{\gamma}}|}{4}-k}{\color{black}(\|v_0\|_{L^\infty(\R^n)}+\|\nabla p\|_{L^\infty(\R^n)})}
 \label{eqn:uniformregularity2}
 \end{align}
and similarly for $v_\lambda = v$,
\begin{equation}
\sup_{\lambda>1}\left\|\partial_t^k\nabla^{{\gamma}}\nabla v_\lambda(\cdot , t)\right\|_{L^\infty}
   \lesssim
t^{-\frac{|{{\gamma}}|}{4}-k}\|v_0\|_{L^\infty(\R^n)}.
\label{v2.est}
\end{equation}

For the $L_\infty$-estimate for $\Phi_\lambda$, we combine
 \eqref{eqn:KerDecay1}, \eqref{eqn:KerDecay2}, \eqref{eqn:uniformregularity2} and \eqref{eqn:WFnonlinear} to give
 \begin{align} \label{eqn:WFunibnd}
  \sup_{\lambda>1}\left\|\Phi_\lambda(\cdot , t)\right\|_{L^\infty(\R^n)}
   \leq& \sup_{\lambda>1}\left\|e^{-\Delta^2 t}p_\lambda(\cdot )\right\|_{L^\infty(\R^n)}\nonumber\\
   &+\sup_{\lambda>1}\left\|\int_{0}^{t}\int_{\R^n}b(\cdot-y,t-s)(f_0[u_\lambda]-f_0[v])(y,s)dyds\right\|_{L^\infty(\R^n)}\nonumber\\
   &+\sup_{\lambda>1}\left\|\int_{0}^{t}\int_{\R^n}\nabla_i b(\cdot -y, t-s)(f_1^i[u_\lambda]-f_1^i[v])(y,s)dyds\right\|_{L^\infty(\R^n)}\nonumber\\&+\sup_{\lambda>1}\left\|\int_{0}^{t}\int_{\R^n}\nabla_{ij}b(\cdot -y, t-s)(f_2^{ij}[u_\lambda]-f_2^{ij}[v])dyds\right\|_{L^\infty(\R^n)}\nonumber\\
   \le&\sup_{\lambda>1}\left\|e^{-\Delta^2 t}p_\lambda(\cdot )\right\|_{L^\infty(\R^n)}\nonumber\\
   &+\sup_{\lambda>1}\left\|\int_{0}^{t}\int_{\R^n}|b(\cdot-y,t-s)|(|f_0[u_\lambda]|+|f_0[v]|)(y,s)dyds\right\|_{L^\infty(\R^n)}\nonumber\\
   &+\sup_{\lambda>1}\left\|\int_{0}^{t}\int_{\R^n}|\nabla_i b(\cdot -y, t-s)|(|f_1^i[u_\lambda]|+|f_1^i[v]|)(y,s)dyds\right\|_{L^\infty(\R^n)}\nonumber\\&+\sup_{\lambda>1}\left\|\int_{0}^{t}\int_{\R^n}|\nabla_{ij}b(\cdot -y, t-s)|(|f_2^{ij}[u_\lambda]|+|f_2^{ij}[v]|)(y,s)dyds\right\|_{L^\infty(\R^n)}\nonumber
\end{align}
Now we make use of the structure for the nonlinear terms 
\eqref{nonlin.struct.est.WF} 
together with the kernel and regularity estimates 
\eqref{eqn:KerDecay3}, \eqref{eqn:uniformregularity2} and \eqref{v2.est}, 
we have
\begin{eqnarray}
& & \left\|\int_{0}^{t}\int_{\R^n}|b(\cdot-y,t-s)|(|f_0[u_\lambda]|+|f_0[v]|)(y,s)dyds\right\|_{L^\infty(\R^n)}\nonumber\\
& \lesssim & 
\int_{0}^{t}\int_{\R^n}|b(y,t-s)|s^{-\frac{3}{4}}\,dy\,ds
\lesssim \int_0^t s^{-\frac{3}{4}}\,ds \lesssim t^{\frac{1}{4}};
\label{N0-est}
\end{eqnarray}
\begin{eqnarray}
& & \left\|\int_{0}^{t}\int_{\R^n}
|\nabla_i b(\cdot-y,t-s)|
(|f_1^i[u_\lambda]|+|f_1^i[v]|)(y,s)dyds\right\|_{L^\infty(\R^n)}\nonumber\\
& \lesssim & 
\int_{0}^{t}\int_{\R^n}|\nabla b(y,t-s)|s^{-\frac{2}{4}}\,dy\,ds\nonumber\\
& \lesssim & \int_{0}^{t}(t-s)^{-\frac{1}{4}}s^{-\frac{2}{4}}\,ds
= t^{\frac{1}{4}}
\int_{0}^{1}(1-s)^{-\frac{1}{4}}s^{-\frac{2}{4}}\,ds \lesssim t^{\frac{1}{4}}
\label{N1-est}
\end{eqnarray}
\begin{eqnarray}
& & \left\|\int_{0}^{t}\int_{\R^n}
|\nabla_{ij} b(\cdot-y,t-s)|
(|f_2^{ij}[u_\lambda]|+|f_2^{ij}[v]|)(y,s)dyds\right\|_{L^\infty(\R^n)}\nonumber\\
& \lesssim & 
\int_{0}^{t}\int_{\R^n}|\nabla^2 b(y,t-s)|s^{-\frac{1}{4}}\,dy\,ds\nonumber\\
& \lesssim & \int_{0}^{t}(t-s)^{-\frac{2}{4}}s^{-\frac{1}{4}}\,ds
= t^{\frac{1}{4}}
\int_{0}^{1}(1-s)^{-\frac{2}{4}}s^{-\frac{1}{4}}\,ds \lesssim t^{\frac{1}{4}}.
\label{N2-est}
\end{eqnarray}
Hence we have
$$
\sup_{\lambda>1}\left\|\Phi_\lambda(\cdot , t)\right\|_{L^\infty(\R^n)}
\lesssim 
\sup_{\lambda >1}\left\|p_\lambda\right\|_{L^\infty}
+ t^{\frac{1}{4}}
\left({{\|\nabla v_0\|_{L^\infty(\R^n)}+\|\nabla p\|_{L^\infty(\R^n)}}}\right)
< \infty.
$$

For higher order regularity estimates, by \eqref{eqn:uniformregularity2}, we have
  \begin{equation}
 \sup_{\lambda>1}\|\nabla^{{\gamma}}\partial_t^k \nabla\Phi_\lambda(\cdot,t)\|_{L^\infty(\R^n)}\leq C t^{-\frac{|{{\gamma}}|}{4}-k}({{\|\nabla v_0\|_{L^\infty(\R^n)}+\|\nabla p\|_{L^\infty(\R^n)}}}).
 \label{eqn:WFunifb2}
 \end{equation} 

 As in the MCF case, we  apply the Arzela--Ascoli theorem to 
extract a subsequence $\left\{ \Phi_{\lambda_k} \right\}$ and $\Phi_\infty(x, t)\in C^\infty({\color{black}\R^n\times(0,1]})$ such that for any $\delta>0$ and any compact subset $K$ of $\R^n$, 
 \begin{equation}
   \lim_{\lambda_k\rightarrow \infty}\sup_{\delta\leq t\leq 1}\left\|\Phi_{\lambda_k}(\cdot , t)-\Phi_\infty(\cdot , t)\right\|_{C^k(K)}=0, \quad \forall k\in \mathbb{N}^+,
   \label{eqn:WFuniformconver}
 \end{equation}
 and $\Phi_\infty$ satisfies the regularity estimate \eqref{eqn:PhiRegularity}. 

\subsubsection{Equation for $\Phi_\infty$}
Here we check that $\Phi_{\infty}$ satisfies \eqref{eqn:fixpointequation}. 
The strategy is similar to the MCF case. 

Recall that $\Phi_\lambda$ satisfies the following identity:
 \begin{align}
  \Phi_\lambda(x,t)
   &=e^{-\Delta^2 t}p_\lambda(x)+\int_{0}^{t}\int_{\R^n}b(x-y, t-s)(f_0[\Phi_\lambda+v]-f_0[v])(y,s)dyds\nonumber\\
   &\qquad-\int_{0}^{t}\int_{\R^n}\nabla_i b(x-y, t-s)(f_1^i[\Phi_\lambda+v]-f_1^i[v])(y,s)dyds\nonumber\\
   &\qquad+\int_{0}^{{{
   }}}\int_{\R^n}\nabla^2_{ij}b(x-y, t-s)(f_2^{ij}[\Phi_\lambda+v]-f_2^{ij}[v])dyds.  
\label{eqn:WFdifferenceeq}
 \end{align}

First, by the $L^1$-bounded of $b(\cdot ,t)$, 
similar to \eqref{eqn:plambdaconvMCF}, we have
 \begin{equation}
   \lim_{\lambda\rightarrow \infty}\|e^{-\Delta^2 t}p_\lambda(\cdot )\|_{L^\infty(\R^n)}\lesssim \lim_{\lambda\RA\infty}\left\|p_\lambda\right\|_{L^\infty(\R^n)}\lesssim {\color{black}\lim_{\lambda\to\infty}\frac{1}{\lambda}\|p\|_{L^\infty(\R^n)}}=0. \label{eqn:plambdaconvWF}
\end{equation}

Second, similar to the previous computations, 
in particular, the derivations of 
\eqref{N0-est}, \eqref{N1-est}, \eqref{N2-est},
the integrals of the nonlinear terms are all bounded by integrands
that are integrable with bounds independent of $\lambda$. 
Hence, \eqref{eqn:fixpointequation} follows from the 
Lebesgue Dominated Convergence Theorem.
We emphasize here again the crucial use of the estimates 
\eqref{nonlin.struct.est.WF} for the nonlinear terms and 
the $L^1$-bounds \eqref{eqn:KerDecay3} for the derivatives of the 
bi-harmonic heat kernel.

\subsection{Proof of $\Phi_\infty=0$}\label{PhiInfty0}
In this section, we will show that the integral equation 
\eqref{eqn:fixpointequation} only admits the zero solution among the class
of functions with small Lipschitz norm. 
This follows from a fixed point type argument.

Motivated by the translation and scaling invariance of the equation,  
the following functions space was introduced in \cite{KochLamm}.
Let  $T>0$.
\begin{enumerate}
\item For MCF with $\alpha=2$, 
  \begin{multline}
    X_T:=\Big\{f(x,t):\R^n\times(0, T)\RA\R \Big|
    \left\|f\right\|_{X_T}:=\sup_{0<t<T}\left\|\nabla f(\cdot , t)\right\|_{L^\infty(\R^n)}\\
    +\sup_{x\in \R^n}\sup_{0<R^2<T}R^{\frac{2}{n+4}}\left\|\nabla^2 f\right\|_{L^{n+4}(B_R(x)\times(R^2/2, R^2))}<\infty. \Big\}
    \label{eqn:MCFXT}
  \end{multline}
\item For WF with $\alpha=4$, 
    \begin{multline}
  X_T:=\Big\{f(x,t):\R^n\times(0, T)\RA\R \Big|\left\|f\right\|_{X_T}=\sup_{0<t<T}\left\|\nabla f(\cdot , t)\right\|_{L^\infty(\R^n)}\\
  +\sup_{x\in \R^n}\sup_{0<R^4<T}R^{\frac{2}{n+6}}\left\|\nabla^2 f\right\|_{L^{n+6}(B_R(x)\times(R^4/2, R^4))}<\infty. \Big\}
  \label{eqn:WFXT}
  \end{multline}
\end{enumerate}
Note that the above norms are scale invariant:
$$
\|f_\lambda\|_{X_T} = \|f\|_{X_{\lambda^{\alpha}T}}
\quad\text{and}\quad
\|f_\lambda\|_{X_\infty} = \|f\|_{X_\infty}.
$$

We then have the following estimate.
\begin{lemma}[Koch--Lamm \cite{KochLamm} Lemma 3.10 and 5.2]\label{lemma:4.1}
For any $0<T\le\infty$ and $0<\delta<1$ there exists $C(\delta)>0$ s.t. 
for every $g_1, g_2\in B_{\delta}^{X_T}(0):=\{g\in X_T|\|g\|_{X_T}\le \delta\}$, we have 
\begin{multline} 
\left\|\int_{0}^{T}e^{-(T-s)A}N[g_1](x,s)ds
-\int_{0}^{T}e^{-(T-s)A}N[g_2](x,s)ds\right\|_{X_T}\\
\leq C(\delta)
(\left\|g_1\right\|_{X_T} +\left\|g_2\right\|_{X_T})
\left\|g_1-g_2\right\|_{X_T}.  
\label{eqn:MCFcontraction}
\end{multline}
\end{lemma}

The above is established through the following linearized 
estimate:
$$
\left\|\int_0^T e^{-(T-s)A}g\,ds\right\|_{X_T}
\leq
\|g\|_{Y_T}
$$
for some appropriate spatial-temporal function space $Y_T$
\cite[Lemma 3.11, 5.3]{KochLamm}.
We will in fact present the proof of the above result in the setting of 
weighted function spaces, $X_T^\beta$ and $Y_T^\beta$ -- 
see Lemmas \ref{lemma:keylemma1} and \ref{lemma:keylemma2}.

We apply the above lemma with $T=1$, $g_1 = \Phi_\infty + v$ and $g_2=v$.
Suppose we can show that 
$\|g_1\|_{X_T}, \|g_2\|_{X_T} \ll 1$, then we would have
\begin{align*}
    \left\|\Phi_\infty\right\|_{X_T}
&=\left\|\int_{0}^{T}e^{-(T-s)A}(N[\Phi_\infty+v]-N[v])(x,s)ds\right\|_{X_T}
   \ll \left\|\Phi_\infty\right\|_{X_T},
\end{align*}
which implies $\left\|\Phi_\infty\right\|_{X_T}=0$.
Hence $\nabla\Phi_{\infty}\equiv0$ leading to $N[\Phi_{\infty}+v]=N[v]$
as $N(\cdot)$ only involves the derivatives of $\Phi_\infty$. 
From \eqref{eqn:fixpointequation}, we conclude that $\Phi_{\infty}\equiv0$.

Hence we are led to compute the $X_T$-norm of $g_1$ and $g_2$ under the
regularity estimates given by
\eqref{eqn:KLregularity} and \eqref{eqn:PhiRegularity}.

For MCF, we have,
\begin{align*}
  &\left\|\Phi_\infty\right\|_{X_T}+\left\|v\right\|_{X_T}\\ 
\lesssim & ({{\|\nabla v_0\|_{L^\infty(\R^n)}+\|\nabla p\|_{L^\infty(\R^n)}}})\left(1+\sup_{0<R^2<T} R^{\frac{2}{n+4}}\left(\int_{B_R(x)\times(R^2/2,R^2)} 
(t^{-\frac{1}{2}})^{n+4}dt dy\right)^{\frac{1}{n+4}}\right)\\
  \lesssim & (\|\nabla v_0\|_{L^\infty(\R^n)}+\|\nabla p\|_{L^\infty(\R^n)})
\left(1+\sup_{0<R^2<T} R^{\frac{2}{n+4}}\left(R^n\int_{R^2/2}^{R^2} 
t^{-\frac{n+4}{2}}dt \right)^{\frac{1}{n+4}}\right)\\
  \lesssim & ({{\|\nabla v_0\|_{L^\infty(\R^n)}+\|\nabla p\|_{L^\infty(\R^n)}}})\left(1+\sup_{0<R^2<T} R^{\frac{2}{n+4}}\left(R^n R^{-n-2} \right)^{\frac{1}{n+4}}\right)\\
  \lesssim &({{\|\nabla v_0\|_{L^\infty(\R^n)}+\|\nabla p\|_{L^\infty(\R^n)}}}).
\end{align*}

For WF, we have,
\begin{align*}
  &\left\|\Phi_\infty\right\|_{X_T}+\left\|v\right\|_{X_T}\\ 
\lesssim & ({{\|\nabla v_0\|_{L^\infty(\R^n)}+\|\nabla p\|_{L^\infty(\R^n)}}})\left(1+\sup_{0<R^4<T} R^{\frac{2}{n+6}}\left(\int_{B_R(x)\times(R^4/2,R^4)} 
(t^{-\frac{1}{4}})^{n+6}dt dy\right)^{\frac{1}{n+6}}\right)\\
  \lesssim & (\|\nabla v_0\|_{L^\infty(\R^n)}+\|\nabla p\|_{L^\infty(\R^n)})
\left(1+\sup_{0<R^4<T} R^{\frac{2}{n+6}}\left(R^n\int_{R^4/2}^{R^4} t^{-\frac{n+6}{4}}dt \right)^{\frac{1}{n+6}}\right)\\
  \lesssim & ({{\|\nabla v_0\|_{L^\infty(\R^n)}+\|\nabla p\|_{L^\infty(\R^n)}}})\left(1+\sup_{0<R^4<T} R^{\frac{2}{n+6}}\left(R^n R^{-n-2} \right)^{\frac{1}{n+6}}\right)\\
  \lesssim &({{\|\nabla v_0\|_{L^\infty(\R^n)}+\|\nabla p\|_{L^\infty(\R^n)}}}).
\end{align*}

The above show that in order to obtain the desired result, we just
need to take the Lipschitz norms of 
$v_0$ and $p$ to be sufficiently small which is indeed assumed to be
the case under the current setting.

\section{Equi-decay and Global Uniform Convergence}\label{sec:equidecay}
Here we will tackle Theorem \ref{thm:GlobalConver}.
In essence, if the gradient of initial perturbation is assumed to have some 
spatial decay, then we can obtain a global-in-space convergence result. 
The idea is to establish the equi-decay property of 
$\{\Phi_\lambda\}_{\lambda>1}$ via a contraction property of the nonlinear 
operators in some weighted spaces. For convenience, we recall here the
weighted Lipschitz seminorm used in Theorem \ref{thm:GlobalConver}:
\begin{equation}
[p]_\beta:=\|(1+|x|^\beta)\nabla p(x)\|_{L^\infty(\R^n)}.
\end{equation}

\subsection{MCF}
For the mean curvature flow case, 
we introduce the following function space which is the spatially
weighted version of $X_T$:
\begin{definition}
For every $0<T<\infty$, we define the function space $X_T^\beta$ by
\begin{multline}
X_T^\beta=\Big\{u\Big|
\|u\|_{X_T^\beta}:=\sup_{0<t<T}\sup_{x\in \R^n}(1+|x|^\beta)|\nabla u(t, x)|\\
+\sup_{x\in\R^n}\sup_{0<R^2<T}(1+|x|^\beta)R^{\frac{2}{n+4}}\|\nabla^2 u\|_{L^{n+4}(Q_R(x) )}<\infty\Big\},
\end{multline}
where 
\begin{equation*}
Q_R(x):=B_R(x)\times(R^2/2,R^2).
\end{equation*}
\end{definition}
Then we have the following linear estimate.
	\begin{lemma} For $k\geq 0$ and $0 < t < T$, 
		\begin{equation}
		\left\|t^{\frac{k}{2}}\nabla^k e^{t\Delta}p(x)\right\|_{X_T^\beta}\lesssim[p]_\beta. \label{eqn:MCFlinearestimate}
		\end{equation}	  
	  \label{lemma:MCFlinearestimate}
	\end{lemma}	

For the analysis of the nonlinear part, we introduce the weighted function spaces $Y_T^\beta$ as follows.
\begin{definition}
  For every $0<T\leq \infty$, we define the function space $Y_T^\beta$ by
  $$Y_T^\beta=\left\{ g\Big|\left\|g\right\|_{Y_T}^\beta=\sup_{x\in\R^n}\sup_{0<R^2<T}(1+|x|^\beta)R^{\frac{2}{n+4}}\left\|g\right\|_{L^{n+4}(Q_R(x))} <\infty\right\}.$$
\end{definition}

Now we define
\begin{equation}
	Sg(x,t):=\int_0^t\int_{\R^n}h(x-y,t-s)g(y, s)dyds..
\end{equation}
The following is the key technical estimate concerning $S$.
\begin{lemma} For $0<t<T<\infty$, 
  $$\sup_{0<t<T}\|(1+|x|^\beta)Sg(x, t)\|_{L^\infty(\R^n)}+\left\|Sg\right\|_{X_{T}^\beta}\lesssim\left\|g\right\|_{Y_T^\beta}.$$
  \label{lemma:keylemma1}
\end{lemma}
With the above, then we have the following result for the nonlinear functional.
\begin{lemma}
  For every $0<T<\infty$, 
  \begin{equation}
    \left\|N[u]-N[v]\right\|_{Y_T^\beta}
    \lesssim
    \left( \left\|u\right\|_{X_T}^2+\left\| v\right\|_{X_T}^2 \right)
    \left\|u-v\right\|_{X_T^\beta}. 
    \label{eqn:MCFNdiff}
  \end{equation}
In particular,
there exist $\varepsilon>0$ and $q<1$ such that for all $[v_0]+[p]_{\beta}<\varepsilon$, 
          \begin{equation}
            \left\|\int_{0}^{t}e^{(t-s)\Delta}(N[u]-N[v])(x, s)ds\right\|_{X_T^\beta}\leq q\left\|u-v\right\|_{X_T^\beta}.
            \label{eqn:4.7}
          \end{equation}
  \label{lemma:kochlamm}
\end{lemma}

We will give the proofs of Lemmas \ref{lemma:MCFlinearestimate} and 
\ref{lemma:kochlamm} here but that for Lemma \ref{lemma:keylemma1}
in the Appendix due to its length and technical nature.

	\begin{proof}[Proof of Lemma \ref{lemma:MCFlinearestimate}]
	  It suffices to show that there exists a $C>0$ depending only on $T$, $n$,  $\beta$ and $k$ such that if $[p]_\beta\leq 1$, then $\left\|e^{t\Delta}p(x)\right\|_{X_T^\beta}\leq C$. From the definition of
	  $\|\cdot\|_{X_T^\beta}$, We need to estimate two terms. 

First, consider
	  \begin{align*}
	  	&|t^{\frac{k}{2}}\nabla^k\nabla e^{t\Delta }p(x)|\\
	  	&=\frac{1}{(4\pi t)^{\frac{n}{2}}}\left|\int_{\R^n}t^{\frac{k}{2}}\nabla^k_x\nabla_x e^{-\frac{|x-y|^2}{4t}}p(y)dy\right|\\
	  	&=\frac{1}{(4\pi t)^{\frac{n}{2}}}\left|\int_{\R^n}t^{\frac{k}{2}}\nabla^k_y e^{-\frac{|x-y|^2}{4t}}\nabla_y p(y)dy\right|\\
	  	&\leq \frac{1}{(4\pi t)^{\frac{n}{2}}}\int_{\R^n}\left|t^{\frac{k}{2}}\nabla_y^ke^{-\frac{|x-y|^2}{4t}}\right||\nabla_y p(y)|dy\\
	  	&\leq\frac{1}{(4\pi t)^{\frac{n}{2}}}\left( \int_{ \left\{ y:|y-x|\leq \frac{\sqrt{t}}{2\sqrt{T}}|x| \right\}}+\int_{ \left\{ y:|y-x|\geq \frac{\sqrt{t}}{2\sqrt{T}}|x| \right\}} \left|{{\mathcal{P}}}_k\left(\frac{x-y}{\sqrt{t}}\right)\right|e^{-\frac{|x-y|^2}{4t}}\frac{1}{1+|y|^\beta}\right)dy\\
	  	&=:\mathrm{I}+\mathrm{II}, 
	  \end{align*}
	  where ${{\mathcal{P}}}_k$ is some polynomial of degree $k$. For ${\rm I}$, $|y-x|\le \frac{\sqrt{t}}{2\sqrt{T}}|x|$ implies that $|y|\ge \frac{|x|}{2}$ for $0<t<T$. Hence, 
	  \begin{equation*}
	  	\frac{1}{1+|y|^\beta}\le \frac{1}{1+|x/2|^\beta}=\frac{2^\beta}{2^\beta+|x|^\beta}\le\frac{2^\beta}{1+|x|^\beta}
	  \end{equation*}
	  so that
	  \begin{align*}
	  	\mathrm{I}&\lesssim\frac{1}{1+|x|^\beta}\int_{ \left\{ y:|y-x|\leq \frac{\sqrt{t}}{2\sqrt{T}}|x| \right\}}\frac{1}{(4\pi t)^{\frac{n}{2}}}\left|{{\mathcal{P}}}_k\left( \frac{x-y}{\sqrt{t}} \right)\right|e^{-\frac{|x-y|^2}{4t}}dy\\
	  	&\lesssim\frac{1}{1+|x|^\beta}\int_{\R^n}|{{\mathcal{P}}}_k(z)|e^{-|z|^2}dz \lesssim \frac{1}{1+|x|^\beta},
	  \end{align*}
	  while for II, when $|y-x|\ge\frac{\sqrt{t}}{2\sqrt{T}}|x|$, we have 
	  \begin{equation*}
	  	e^{-\frac{|x-y|^2}{4t}}=e^{-\frac{|x-y|^2}{8t}}e^{-\frac{|x-y|^2}{8t}}\le e^{-\frac{|x|^2}{32T}}e^{-\frac{|x-y|^2}{8t}},
	  \end{equation*}
	  so that
	  \begin{align*}
	  	\mathrm{II}&\leq e^{-\frac{|x|^2}{32T}}\int_{ \left\{ y:|y-x|\geq \frac{\sqrt{t}}{2\sqrt{T}}|x| \right\}}\frac{1}{(4\pi t)^{\frac{n}{2}}}e^{-\frac{|x-y|^2}{8t}}\left|{{\mathcal{P}}}_k\left( \frac{x-y}{2\sqrt{t}} \right)\right|dy\\
	  	&\lesssim e^{-\frac{|x|^2}{32T}}\int_{\R^n}|{{\mathcal{P}}}_k(z)|e^{-\frac{|z|^2}{2}}dz
	  	\lesssim e^{-\frac{|x|^2}{32T}}\lesssim\frac{1}{1+|x|^\beta}.
	  \end{align*}
	  Combining I and II, we have
	  \begin{equation}
	  	\left|t^{\frac{k}{2}}\nabla^k{{\nabla}} e^{t\Delta} p(x)\right|
\lesssim \frac{1}{1+|x|^\beta}.
	  	\label{eqn:MCFlinearpointwiseestimate}
	  \end{equation} 

	  Second, we estimate
	  \begin{eqnarray}\label{est.sec}
	  	&&\sup_{x\in\R^n}\sup_{0<R^2<T}(1+|x|^\beta)R^{\frac{2}{n+4}}\left\|t^{\frac{k}{2}}\nabla^k\nabla^2e^{t\Delta} p(x)\right\|_{L^{n+4}(Q_R(x))}.
	  \end{eqnarray}
Note that
\begin{eqnarray*}
& & \left\|t^{\frac{k}{2}}\nabla^k\nabla^2e^{t\Delta} p(x)\right\|^{n+4}_{L^{n+4}(Q_R(x))}\\
& = & 
\int_{R^2/2}^{R^2}\int_{B_R(x)}\left[
t^{\frac{k}{2}}\nabla^k\nabla^2\int 
\frac{1}{(4\pi t)^\frac{n}{2}}e^{-\frac{|y-z|^2}{4t}}p(z)\,dz
\right]^{n+4}\,dy\,dt\\
& = &
\int_{R^2/2}^{R^2}\int_{B_R(x)}\left[
t^{\frac{k}{2}}\int 
\frac{1}{(4\pi t)^\frac{n}{2}}\nabla^{k+1}e^{-\frac{|y-z|^2}{4t}}\nabla p(z)\,dz
\right]^{n+4}\,dy\,dt\\
& \lesssim &
\int_{R^2/2}^{R^2}\int_{B_R(x)}\left[
t^{\frac{k}{2}}\int 
\frac{1}{(4\pi t)^\frac{n}{2}}t^{-\frac{k+1}{2}}
e^{-\frac{|y-z|^2}{4t}}{\cal P}_{k+1}\left(\frac{y-z}{\sqrt{t}}\right)
\frac{1}{1+|z|^\beta}\,dz
\right]^{n+4}\,dy\,dt\\
& \lesssim &
\int_{R^2/2}^{R^2}\int_{B_R(x)}\left[
t^{-\frac{1}{2}}\int 
\frac{1}{(4\pi t)^\frac{n}{2}}
e^{-\frac{|y-z|^2}{4t}}{\cal P}_{k+1}\left(\frac{y-z}{\sqrt{t}}\right)
\frac{1}{1+|z|^\beta}\,dz
\right]^{n+4}\,dy\,dt\\
& \lesssim &
\int_{R^2/2}^{R^2}\int_{B_R(x)}\left[
\frac{t^{-\frac{1}{2}}}
{1+|y|^\beta}\right]^{n+4}\,dy\,dt\\
&\lesssim&
\int_{R^2/2}^{R^2}t^{-\frac{1}{2}(n+4)}dt\int_{B_R(x)}\frac{1}{(1+|y|^\beta)^{n+4}}dy\\
&\lesssim &R^{-(n+2)} |B_R(x)| \frac{1}{(1+|x|^{\beta})^{n+4}} \\
&\lesssim &\frac{R^{-2}}{(1+|x|^\beta)^{n+4}}
\end{eqnarray*}
which leads to that $\text{\eqref{est.sec}}\lesssim 1$.

The above two parts combined give 
$\|t^{\frac{k}{2}}\nabla^k e^{t\Delta} p(x)\|_{X_T^\beta}\le C$. 
	  	\end{proof}
\begin{proof}[Proof of Lemma \ref{lemma:kochlamm}]
Recall the form \eqref{eqn:GrapMCF} for the nonlinear term $N(u)$. 
First note that
\begin{equation*}
  \left|(1+|\nabla u|^2)^{-1}-(1+|\nabla v|^2)^{-1}\right|\le \frac{(|\nabla u|+|\nabla v|)|\nabla (u-v)|}{(1+|\nabla u|^2)(1+|\nabla v|^2)}.
\end{equation*}
Then we have
\begin{eqnarray*}
  & & |N[u]-N[v]|\\
  &=&\left|(1+|\nabla u|^2)^{-1} \nabla u\star \nabla u\star \nabla^2 u-(1+|\nabla v|^2)^{-1}\nabla v\star \nabla v\star \nabla^2 v\right|\\
  &\lesssim&(|\nabla u|+|\nabla v|)(|\nabla^2 u|+|\nabla^2 v|)|\nabla(u-v)|
 +(|\nabla  u|+|\nabla  v|)^2|\nabla^2 (u-v)|.
\end{eqnarray*}
Then estimate \eqref{eqn:MCFNdiff} follows from
\begin{eqnarray*}
	&&\|N[u]-N[v]\|_{Y_T^\beta}\\
	&=&\sup_{x\in\R^n}\sup_{0<R^2<T}(1+|x|^\beta)R^{\frac{2}{n+4}}\|N[u]-N[v]\|_{L^{n+4}(Q_R(x))}\\
	&\lesssim&\sup_{0<t<T}(\|\nabla u\|_{L^\infty(\R^n)}+\|\nabla v\|_{L^\infty(\R^n)})\times\\
	&&\sup_{x\in\R^n}\sup_{0<R^2<T}R^{\frac{2}{n+4}}\left(\|\nabla^2 u\|_{L^{n+4}(Q_R(x))}+\|\nabla^2 v\|_{L^{n+4}(Q_R(x))}\right)\times\\
	&&\sup_{0<t<T}\sup_{x\in\R^n}(1+|x|^\beta)|\nabla(u-v)|\\
	&&+\sup_{0<t<T}\left( \|\nabla u\|_{L^\infty(\R^n)}+\|\nabla v\|_{L^\infty(\R^n)} \right)^2\sup_{x\in\R^n}(1+|x|^\beta)\sup_{0<R^2<T}R^{\frac{2}{n+4}}\|\nabla^2(u-v)\|_{L^{n+4}(Q_R(x))}\\
	&\lesssim&(\|u\|_{X_T}+\|v\|_{X_T})^2\|u-v\|_{X_T^\beta}. 
\end{eqnarray*} 
For \eqref{eqn:4.7}, using Lemma \ref{lemma:keylemma1}, we have that 
\begin{align*}
	   \left\| S(N[u]-N[v])\right\|_{X_T^\beta}
	   &
	   \lesssim \left\|N[u]-N[v]\right\|_{Y_{T}^\beta}\\
	   &\lesssim\left(\left\|u\right\|_{X_T}^2+\left\|v\right\|_{X_T}^2\right)\left\|u-v\right\|_{X_T^\beta}\\
	   &\lesssim([p]^2+[v_0]^2)\left\|u-v\right\|_{X_T^\beta}
	 \lesssim\varepsilon^2\left\|u-v\right\|_{X_T^\beta}.
	 \end{align*}
Note that we have used unweighted version Theorem \ref{thm:KochLammRegu} to
estimate the $\|\cdot\|_{X_T}$ norms by the initial data.
Hence, \eqref{eqn:4.7} holds if we take $\varepsilon$ to be sufficiently small. 
\end{proof}

\subsection{WF case}
The {{strategy}} here is similar to the MCF case. 
We again introduce the following weighted function space:
\begin{multline*}
  X_T^\beta=\Big\{u\left| \left\|u\right\|_{X_T^\beta}:=\right.\\
\sup_{0<t<T}\sup_{x\in \R^n}(1+|x|^\beta)|\nabla u(x, t)|+\sup_{x\in \R^n}\sup_{0<R^4<T}(1+|x|^\beta)R^{\frac{2}{n+6}}\left\|\nabla^2 u\right\|_{L^{n+6}(Q_R(x))}
< \infty\Big\},
\end{multline*}
where $Q_R(x):=B_R(x)\times(R^4/2, R^4).$
\begin{lemma} For $k\geq 0$,
  \begin{equation}
 \left\|t^{\frac{k}{4}}\nabla^ke^{-t\Delta^2}p(x)\right\|_{X_T^\beta}\lesssim [p]_\beta.
    \label{eqn:5.22}
  \end{equation}
  \label{Lemma:linearestimateMCF}
\end{lemma}

Anticipating the forms of the nonlinear terms in \eqref{eqn:GrapWF}, 
we introduce the following weighted function spaces $Y_{0, T}^\beta$, $Y_{1, T}^\beta$ and $Y_{2, T}^\beta$,
where
\begin{align*}
  \left\|g_0\right\|_{Y_{0, T}^\beta}&=\sup_{x\in\R^n}\sup_{0<R^4<T}(1+|x|^\beta)R^{\frac{6}{n+6}}\left\|g_0\right\|_{L^{\frac{n+6}{3}}(Q_R(x))}, \\
  \left\|g_1\right\|_{\YBB}&=\sup_{x\in\R^n}\sup_{0<R^4<T}(1+|x|^\beta)R^{\frac{4}{n+6}}\left\|g_1\right\|_{L^{\frac{n+6}{2}}(Q_R(x))},\\
  \left\|g_2\right\|_{\YBBB}&=\sup_{x\in\R^n}\sup_{0<R^4<T}(1+|x|^\beta)R^{\frac{2}{n+6}}\left\|g_2\right\|_{L^{n+6}(Q_R(x))}.
\end{align*}
Now consider the following operator:
\begin{equation}
  Sg(x, t):=\int_0^t e^{-(t-s)\triangle^2}g\,ds = 
\int_{0}^{t}\int_{\R^n}b(x-y, t-s)g(y,s)dyds.
  \label{}
\end{equation}
The key estimate is the following lemma:
\begin{lemma} For every $0<t<T<\infty$, 
 \begin{equation}
   \sum_{l=0}^{2}\left( \sup_{0<t<T}\|(1+|x|^\beta)\nabla^l Sg_l(\cdot , t)\|_{L^\infty(\R^n)}+\left\|\nabla^l Sg_l\right\|_{X_T^\beta} \right)\lesssim\sum_{l=0}^{2}\left\|g_l\right\|_{Y_{l, T}^\beta}.
   \label{}
 \end{equation}
  \label{lemma:keylemma2}
\end{lemma}

\begin{lemma}
  For every $0<T<\infty$, 
  \begin{align}
    \sum_{l=0}^{2}\left\| (f_l[u]-f_l[v])\right\|_{Y_{l, T}^\beta}
    \lesssim\left(\left\| u\right\|_{X_T}+\left\|v\right\|_{X_T}\right)\left\|u-v\right\|_{X_T^\beta}
    \label{WF.nonlinear}
  \end{align}
(Recall the forms \eqref{f0}--\eqref{f2} for the $f_l$'s.)
In particular, there exist $\varepsilon>0$ and $q<1$ such that for all $[v_0]+[p]<\varepsilon$,
  \begin{equation}\label{lemma:kochlamm2-eqn}
 \sum_{l=0}^{2}
 \left\|
 \int_0^t e^{-(t-s)\Delta^2}
 \left(\nabla^l f_l(u)-\nabla^l f_l(v)\right)
 \,ds
 \right\|_{X_T^\beta}\leq q\left\|u-v\right\|_{X_T^\beta}. 
\end{equation}
  \label{lemma:kochlamm2}
\end{lemma}

\begin{proof}[Proof of Lemma \ref{Lemma:linearestimateMCF}]
  It suffices to show that there exists a $C>0$ depending only on $T, n,\beta$ and $k$ such that if $[p]_\beta\leq 1$, then $\left\|e^{-t\Delta^2}p(x)\right\|_{X_T^\beta}\leq C$. Again, we need to estimate two terms.

First, by the estimate \eqref{eqn:KerDecay2} for the biharmonic kernel $b$, 
for any $k\in \mathbb{N}_+$, there exists $c_k > 0$ such that
  \begin{eqnarray*}
    \Big|t^{\frac{k}{4}}\nabla^k \nabla e^{-t\Delta^2}p(x)\Big|
    &=& \left|\int_{\R^n}t^{\frac{k}{4}}\nabla^k_x \nabla_x b(t, x-y)p(y)dy\right|
    \leq \int_{\R^n}\left|t^{\frac{k}{4}}\nabla^k_y b(x-y, t)\right||\nabla_y p(y)|dy\\
    &\lesssim&\left(\int_{ \left\{ y:|y-x|\leq \frac{t^{\frac{1}{4}}}{2T^{\frac{1}{4}}}|x| \right\}}+\int_{ \left\{ y:|y-x|\geq \frac{t^{\qua}}{2T^{\frac{1}{4}}}|x| \right\}}t^{-\frac{n}{4}}{{e^{-c_k\left|(x-y)t^{-\qua}\right|^{\frac{4}{3}}}}}\frac{1}{1+|y|^\beta}dy\right)\\
    &=:&{{\rm I}}+{{\rm II}}, 
  \end{eqnarray*}
where similar to the MCF case, we have
  \begin{eqnarray*}
   {{ {\rm I}}}&\lesssim& \frac{1}{1+|x|^\beta}\int_{ \left\{ y:|y-x|\leq \frac{t^{\qua}}{2T^{\frac{1}{4}}}|x| \right\}}t^{-\frac{n}{4}}e^{-{{c_k}}\left|(x-y)t^{-\qua}\right|^{\frac{4}{3}}}dy
    \lesssim\XB,\\
    {{\rm II}} &\lesssim& e^{-C|x|^{\frac{4}{3}}}\int_{ \left\{ y:|y-x|\geq \frac{t^\qua}{2T^\qua}|x| \right\}}t^{-\frac{n}{4}}e^{-\frac{{{c_k}}}{2}\left|(x-y)t^{-\qua}\right|^{\frac{4}{3}}}dy
\lesssim \XB
  \end{eqnarray*}
so that
\begin{equation}
  \sup_{0< t< T}\sup_{x\in\mathbb{R}^2}
(1+|x|^\beta)\left|t^{\frac{k}{4}}\nabla^k{{\nabla}} e^{-t\Delta^2} p(x)\right|\lesssim 
1.
\label{WF.lin1}
\end{equation}
Second, we compute
\begin{eqnarray*}
\Big\|t^{\frac{k}{4}}\nabla^k \nabla^2 e^{-t\Delta^2}p(x)\Big\|_{L^{n+6}(Q_R(x))}& = & 
\int_{R^4/2}^{R^4}\int_{B_R(x)} \left[
t^{\frac{k}{4}}\nabla^{k+1}b(y-z,t)\nabla p(z)
\,dz\right]^{n+6}\,dy\,dt\\
& \lesssim & 
\int_{R^4/2}^{R^4}\int_{B_R(x)} \left[
t^{-\frac{1}{4}}
t^{-\frac{n}{4}}
e^{-{{c_k}}\left|(y-z)t^{-\qua}\right|^{\frac{4}{3}}}
\frac{1}{1+|z|^\beta}
\,dz\right]^{n+6}\,dy\,dt\\
& \lesssim & 
\int_{R^4/2}^{R^4}\int_{B_R(x)} \left[
\frac{t^{-\frac{1}{4}}}{1+|y|^\beta}
\,dz\right]^{n+6}\,dy\,dt\\
& \lesssim & 
\int_{R^4/2}^{R^4}
t^{-\frac{n+6}{4}}
\,dt
\int_{B_R(x)} \frac{1}{(1+|y|^\beta)^{n+6}}\,dy\\
& \lesssim & 
R^{-2}\frac{1}{(1+|x|^\beta)^{n+6}}.
\end{eqnarray*}
which implies that
\begin{equation}
\sup_{x\in \R^n}\sup_{0<R^4<T}(1+|x|^\beta)R^{\frac{2}{n+6}}
\left\|
t^{\frac{k}{4}}\nabla^k \nabla^2 e^{-t\Delta^2}p
\right\|_{L^{n+6}(Q_R(x))}
\lesssim 1.
\label{WF.lin2}
\end{equation}

Combining \eqref{WF.lin1} and \eqref{WF.lin2} then gives
Lemma \ref{Lemma:linearestimateMCF}. 
\end{proof}

\begin{proof}[Proof of Lemma \ref{lemma:kochlamm2}]
It is similar to that of Lemma \ref{lemma:kochlamm}. 
We will just highlight some key computations, 
though mostly at the symbolic level.

Recall the form of $f_0$:
$f_0(u) = (\nabla^2 u)^3{\cal P}(\nabla u)$ where
$\cal P$ is some polynomial. Then
\begin{eqnarray*}
f_0(u) - f_0(v)
& = & 
\left((\nabla^2 u)^3 - (\nabla v)^3\right){\cal P}(\nabla u)
+ (\nabla ^2 v)^3\left({\cal P}(\nabla u) - {\cal P}(\nabla v)\right)\\
& \approx & 
{\cal P}(\nabla u)
\left((\nabla^2 u)^2 + (\nabla v)^2\right)(\nabla^2(u-v))
+ (\nabla^2 v)^3{\cal P}'(\nabla u)(\nabla (u-v))
\end{eqnarray*}
so that
\begin{eqnarray*}
& & \|f_0(u) - f_0(v)\|_{L^{\frac{n+6}{3}}(Q_R(x))}\\
& \lesssim &
\|{\cal P}(\nabla u)\|_{L^\infty(\mathbb R^n)}
\|\left((\nabla^2 u)^2 + (\nabla v)^2\right)\nabla^2(u-v)\|_{L^{\frac{n+6}{3}}(Q_R(x))}\\
& & + \|(\nabla^2 v)^3\|_{L^{\frac{n+6}{3}}(Q_R(x))}
\|{\cal P}'(\nabla u)\|_{L^\infty(\mathbb R^n)}
\|\nabla (u-v)\|_{L^\infty(\mathbb R^n)}\\
& \lesssim &
\|{\cal P}(\nabla u)\|_{L^\infty(\mathbb R^n)}
\left(\|\nabla^2 u\|^2_{L^{n+6}(Q_R(x))}
+ \|\nabla^2 v\|^2_{L^{n+6}(Q_R(x))}\right)
\|\nabla^2(u-v)\|_{L^{n+6}(Q_R(x))}\\
& & + 
\|\nabla^2 u\|^3_{L^{n+6}(Q_R(x))}
\|{\cal P}'(\nabla v)\|_{L^\infty(\mathbb R^n)}
\|\nabla (u-v)\|_{L^\infty(\mathbb R^n)}
\end{eqnarray*}
and hence
\[
\|f_0(u) - f_0(v)\|_{Y_{0,T}^\beta}
\lesssim
\left(\|u\|_{X_T}^2 + \|v\|_{X_T}^2\right)\|u-v\|_{X_T^\beta}.
\]

Similarly, for 
$f_1(u) = (\nabla^2 u)^2{\cal P}(\nabla u)$
and $f_2(u) = (\nabla^2 u){\cal P}(\nabla u)$, we have
\begin{eqnarray*}
& & \|f_1(u) - f_1(v)\|_{L^{\frac{n+6}{2}}(Q_R(x))}\\
& \lesssim &
\|{\cal P}(\nabla u)\|_{L^\infty(\mathbb R^n)}
\left(\|\nabla^2 u\|_{L^{n+6}(Q_R(x))}
+ \|\nabla^2 v\|_{L^{n+6}(Q_R(x))}\right)
\|\nabla^2(u-v)\|_{L^{n+6}(Q_R(x))}\\
& & +
\|\nabla^2 u\|^2_{L^{n+6}(Q_R(x))}
\|{\cal P}'(\nabla v)\|_{L^\infty(\mathbb R^n)}
\|\nabla (u-v)\|_{L^\infty(\mathbb R^n)}
\end{eqnarray*}
and 
\begin{eqnarray*}
\|f_2(u) - f_2(v)\|_{L^{n+6}(Q_R(x))}
& \lesssim &
\|{\cal P}(\nabla u)\|_{L^\infty(\mathbb R^n)}
\|\nabla^2(u-v)\|_{L^{n+6}(Q_R(x))}\\
&&+
\|\nabla^2 u\|_{L^{n+6}(Q_R(x))}
\|{\cal P}'(\nabla v)\|_{L^\infty(\mathbb R^n)}
\|\nabla (u-v)\|_{L^\infty(\mathbb R^n)}
\end{eqnarray*}
so that
\begin{eqnarray*}
\|f_1(u) - f_1(v)\|_{Y_{1,T}^\beta},\,\,\,
\|f_2(u) - f_2(v)\|_{Y_{2,T}^\beta}
\lesssim
\left(\|u\|_{X_T}^2 + \|v\|_{X_T}^2\right)\|u-v\|_{X_T^\beta}
\end{eqnarray*}
and hence completing the proof of \eqref{WF.nonlinear}.

\end{proof}

Again, we postpone the proof of 
Lemma \ref{lemma:keylemma2} to the Appendix due to its technicality.

\subsection{Conclusion of the Proof of Theorem \ref{thm:GlobalConver}}
For simplicity, we just write down the steps for WF as it involves more terms.
Recall the equation for $\Phi_\lambda$:
\begin{equation}
  \Phi_\lambda=e^{-\Delta^2 t}p_\lambda+\sum_{l=0}^{2}(\mathcal{N}_l[v+\Phi_\lambda]-\mathcal{N}_l[v])
  \label{eqn:4.15}
\end{equation}
where 
$$
\mathcal{N}_l(g) = \int_0^t e^{-\Delta^2(t-s)}\nabla^l f_l(g)\,ds.
$$
First, taking the $X_T^\beta$ {{norm}} of both sides of the equation, 
by Lemma \ref{lemma:keylemma2} and 
\eqref{lemma:kochlamm2-eqn} of Lemma \ref{lemma:kochlamm2}, we get
\begin{align*}
  \left\|\Phi_\lambda\right\|_{X_T^\beta}
&\leq \left\|e^{-\Delta^2 t}p_\lambda\right\|_{X_T^\beta}+\sum_{l=0}^{2}\left\|\mathcal{N}_l[\Phi_\lambda+v]-\mathcal{N}_l[v]\right\|_{X_T^\beta}\\
&\leq \left\|e^{-\Delta^2 t}p_\lambda\right\|_{X_T^\beta}
+\sum_{l=0}^{2}\left\|f_l(\Phi_\lambda+v)-f_l[v]\right\|_{Y_{l,T}^\beta}\\
  &\leq \left\|e^{-\Delta^2 t}p_\lambda\right\|_{X_T^\beta}+q\left\|\Phi_\lambda\right\|_{X_T^\beta}.
\end{align*}
Hence, upon choosing $[v_0],\,[p] < \varepsilon$ small enough, we will have $q < 1$ 
which implies a uniform bound for $\Phi_\lambda$ in $X_T^\beta$. More precisely,
\begin{equation}
  \left\|\Phi_\lambda\right\|_{X_T^\beta}\lesssim\left\|e^{-\Delta^2 t}p_\lambda\right\|_{X_T^\beta} \lesssim [p_\lambda]_\beta.
\label{PhiLambdaEst}
\end{equation}
Second, from Lemmas \ref{lemma:keylemma2} and \ref{lemma:kochlamm2} again, 
we have that
\begin{eqnarray*}
&&  \sum_{l=0}^{2}\sup_{0<t<T}\|(1+|x|^\beta)(\mathcal{N}_l[\Phi_\lambda+v]-\mathcal{N}_l[v])\|_{L^\infty(\R^n)}\\
& \lesssim & 
\sum_{l=0}^2 \left\|f_l(\Phi_\lambda+v) - f_l(v)\right\|_{Y_{l,T}^\beta}
\lesssim
\left\|\Phi_\lambda\right\|_{X_T^\beta}\lesssim[p_\lambda]_\beta.
\end{eqnarray*}
When $\lambda>1$, we have $[p_\lambda]_\beta\leq [p]_\beta$. Hence
\begin{equation}
  \sup_{\lambda>1}\sum_{l=0}^{2}\|(1+|x|^\beta)(\mathcal{N}_l[\Phi_\lambda+v]-\mathcal{N}_l[v])(x , T)\|_{L^\infty(\R^n)}\lesssim[p]_\beta.
  \label{InftyTDecayEst}
\end{equation}

With the above, we can prove the global $C^1$-convergence. 
Upon setting $T=1$ in \eqref{InftyTDecayEst}, 
we have the 
$\left\{ \Phi_\lambda(\cdot , 1)-e^{-\Delta^2 }p_\lambda(\cdot )
= \sum_{l=0}^2\mathcal{N}_l(\Phi_\lambda+v)-\mathcal{N}_l(v)\right\}_{\lambda>1}$ 
satisfies the equi-decay property, i.e.
$$
\lim_{R\rightarrow\infty}\sup_{\lambda > 0}
\sup_{|x|<R}\Big|\Phi_\lambda(x,1)-e^{-\Delta^2 }p_\lambda(x,1 )\Big|
=0.
$$
From \eqref{eqn:WFunifb2} (with $\gamma=k=0$) and \eqref{WFp} 
(with the latter applied to $\nabla p_\lambda$)
we have
$$
\Big\|\nabla\big(\Phi_\lambda(\cdot,1)-e^{-\Delta^2 }p_\lambda(\cdot,1)\big)\Big\|_{L^\infty(\mathbb{R}^n)}
\leq
\Big\|\nabla\Phi_\lambda(\cdot,1)\Big\|_{L^\infty(\mathbb{R}^n)} 
+ \Big\|e^{-\Delta^2 }\nabla p_\lambda(\cdot,1)\Big\|_{L^\infty(\mathbb{R}^n)}
< \infty.
$$
Finally, recall \eqref{eqn:plambdaconvWF}. Hence by Arzela-Ascoli Theorem, 
we can conclude that $\Phi_{\lambda_j} \longrightarrow \Phi_\infty$ 
in $C^0(\mathbb{R}^n)$ for a subsequence
$\lambda_j\rightarrow\infty$. The proof of $\Phi_\infty\equiv0$ is the same as
in Section \ref{PhiInfty0} for the spatially un-weighted case. 

For the convergence of $\nabla\Phi_\lambda$, by \ref{PhiLambdaEst}, we have that
$\nabla\Phi_\lambda$ is equi-decay, i.e
$$
\lim_{R\rightarrow\infty}\sup_{\lambda > 0}
\sup_{|x|<R}\Big|\nabla\Phi_\lambda(x,1)\Big|
=0.
$$
From \eqref{eqn:WFunifb2} (with $\gamma=1,k=0$), we further have,
$$
\sup_{\lambda>0}
\Big\|\nabla^2\Phi_\lambda(\cdot,1)\Big\|_{L^\infty(\R^n)} < \infty.
$$
Hence, we deduce that
$\nabla\Phi_{\lambda_j}\longrightarrow\nabla\Phi_\infty\equiv0$ uniformly in 
$\mathbb{R}^n$.

The overall $C^1$-convergence of $u_\lambda=\Phi_\lambda+v$  to $v$ is thus 
established.

\section{Generalization to Polyharmonic Flows}\label{sec:polyharm}
As a future perspective and direction, we use this section to illustrate the 
robustness of the current approach and outline an abstract framework for the 
stability of self-similar solutions to possible higher order polyharmonic flows. 
Suppose the polyharmonic flow, in the graphical setting, takes the following form
\begin{equation}
  \left\{
    \begin{array}{ll}
      \pa_t u+A u=N[u], &\text{ on }\R^n\times(0, \infty), \\
      u(x, 0)=u_0(x), &\text{ in }\R^n, 
    \end{array}
  \right.
  \label{eqn:polyhar}
\end{equation}
where $A=(-\Delta)^m$, $m\ge 2$,  and $N[u]$ is the nonlinear term
-- see \cite{HW} for an example of the form of $N$. Furthermore, we assume that 
\eqref{eqn:polyhar} is invariant under the rescaling 
\begin{equation}
  u_\lambda:=\frac{1}{\lambda}u(\lambda x, \lambda^{2m}t).
  \label{eqn:polyscaling}
\end{equation}
Then for the self-similar initial data $v_0(x)=\lambda^{-1}v_0(\lambda x)$ with small Lipschitz norm, we expect the existence of a self-similar solution $v(x, t)$ to \eqref{eqn:polyhar}, i.e., 
\begin{equation*}
  v(x, t)=v_{t^{-\frac{1}{2m}}}(x, t)=t^{\frac{1}{2m}}v(xt^{-\frac{1}{2m}}, 1)=:t^{\frac{1}{2m}}\Psi(xt^{-\frac{1}{2m}}). 
\end{equation*}

One could follow Koch--Lamm's method to find a unique analytic solution to \eqref{eqn:polyhar} with initial data of small Lipschitz norm in the following scale invariant function space:
\begin{multline}
  X_T:=\Big\{f(x, t):\R^n\times(0, T)\to \R\Big|\|f\|_{X_T}:=
    \sum_{k=0}^{m-2}\sup_{0<t<T}t^{\frac{k}{2m}}\|\nabla^k \nabla f(x , t)\|_{L^\infty(\R^n)}\\
  +\sup_{x\in\R^n}\sup_{0<R^{2m}<T}R^{\frac{(m-1)p-n-2m}{p}}\|\nabla^{m}f\|_{L^p(B_R(x)\times(R^{2m}/2, R^{2m}))}<\infty\Big\} \text{ for some $p>n+2m$. }
  \label{eqn:polyfunspace} 
\end{multline}
We anticipate that a similar procedure as in this paper can show the stability of the self-similar solution $v$ under bounded (and small) perturbation, more specifically, for $u_0=v_0(x)+p(x)$ with $\|p\|_{L^\infty(\R^n)}<\infty$ and $\|\nabla p\|_{L^\infty(\R^n)}<\varepsilon$, it holds that 
\begin{equation}
  \lim_{t\to \infty}\left\|t^{-\frac{1}{2m}}u(t^{\frac{1}{2m}}x, t)-\Psi(x)\right\|_{C^k_{\rm loc}(\R^n)}=0,\quad \forall k\in\mathbb{N}^+. 
  \label{}
\end{equation}
Moreover, by putting the difference $u-v$ in the following weighted space:
\begin{multline}
  X_T^\beta:=\Big\{ f(x, t):\R^n\times(0, T)\to \R\Big| \|f\|_{X_T^\beta}:=\sum_{k=0}^{m-2}\sup_{0<t<T}t^{\frac{k}{2m}}\|(1+|x|^\beta)\nabla^{k}\nabla f(x, t)\|_{L^\infty(\R^n)}\\
  +\sup_{x\in \R^n}\sup_{0<R^{2m}<T}(1+|x|^\beta)R^{\frac{(m-1)p-n-2m}{p}}\|\nabla^m f\|_{L^p(B_R(x)\times(R^{2m}/2, R^{2m}))}<\infty\Big\}.
  \label{}
\end{multline}
we gain the equi-decay property which leads to the global convergence 
\begin{equation}
  \lim_{t\to\infty}\left\|t^{-\frac{1}{2m}}u(t^{\frac{1}{2m}}x, t)-\Psi(x)\right\|_{C^k(\R^n)}=0
  \label{}
\end{equation}
provided the perturbation is small in the weighted space, i.e., $\|(1+|x|^\beta)\nabla p\|_{L^\infty(\R^n)}<\varepsilon$.

\appendix
\section{Proof of Lemma \ref{lemma:keylemma1}}\label{apped:keylemma1}
Before the proof, we first recall some $L^p$-estimates concerning the heat kernel
\eqref{heat.kernel} $h(x,t)$:
for $0<t<\infty$, 
  \begin{align}
    &\left\|h\right\|_{L^p(\R^n\times(0, t))}\lesssim t^{\frac{(n+2)-pn}{2p}},
    \quad\text{for } 1\le p<\frac{n+2}{n},
    \label{eqn:Lpheatkernel}
\\
    &\left\|\nabla h\right\|_{L^p(\R^n\times(0, t))}\lesssim t^{\frac{(n+2)-(n+1)p}{2p}},
        \quad\text{for } 1\le p<\frac{n+2}{n+1},
    \label{eqn:LpGradientHeatKernel}\\
&\left\|\int_{0}^{t}\int_{\R^n}\nabla^2 h(z-y, t-s)g(y, s)dyds\right\|_{L^p(\R^n\times \R^+)}\lesssim\left\|g\right\|_{L^p(\R^n\times \R^+)},
        \quad\text{for } 1\le p< \infty,
    \label{eqn:singintheat}
  \end{align}
where the last is from the theory of singular integral \cite{Stein}. The following
pointwise esimate will also be used:
for all $(z, s)\in \R^n\times(0, t)\setminus B_{\sqt}(0)\times(0,
\frac{t}{2})$, it holds that
    \begin{equation}
      |h(z,s)|+\sqrt{t}|\nabla h(z, s)|+t|\nabla^2 h(z, s)|\leq C t^{-\frac{n}{2}}\exp\left(-c\frac{|z|}{\sqrt{t}}\right)
      \label{eqn:heatkernelexpdecay}
    \end{equation}
which follows from the scaling property of the heat kernel.

\begin{proof}[Proof of Lemma \ref{lemma:keylemma1}]
  It suffices to show that if $\left\|g\right\|_{Y_T^\beta}\leq 1$, then 
\[
\sup_{0<t<T}\|(1+|x|^\beta) Sg(x, t)\|_{L^\infty(\R^n)}+\|Sg\|_{X_T^\beta}\lesssim 1.
\]
For this purpose, we need to estimate
$|Sg(x,t)|$, $|\nabla Sg(x,t)|$, and $\|\nabla^2Sg\|_{L^{n+4}(Q_R(x))}$.
We recall the notation 
$Q_{R}(x) = B_{R}(x)\times(\frac{R^2}{2},R^2)$
and further let $Q'_{R}(x):=B_{R}(x)\times(0, \frac{R^2}{2})$. 
Without loss of generality, we fix $T=1$. 

{\bf Estimate for $Sg$.} We decompose
  \begin{align*}
    |Sg(x, t)|&=\left|\int_{0}^{t}\int_{\R^n}h(x-y, t-s)g(y, s)dyds\right|\\
    &\le \int_{Q_{\sqrt{t}}(x)}+\int_{\R^n\times(0, t)\setminus Q_{\sqrt{t}}(x)}| h(x-y, t-s)g(y, s)|dyds\\
    &:=I_1+I_2.
  \end{align*}

  For $I_1$, by H\"{o}lder inequality and the heat kernel estimate \eqref{eqn:Lpheatkernel} with $p=\frac{n+4}{n+3}< \frac{n+2}{n}$, we have,
  \begin{align}
    I_1&\le \|h\|_{L^{\frac{n+4}{n+3}}(Q'_{\sqrt{t}}(0))}\|g\|_{L^{n+4}(Q_{\sqrt{t}}(x))}
    \le \|h\|_{L^{\frac{n+4}{n+3}}(\R^n\times(0, t/2))}\|g\|_{L^{n+4}(Q_{\sqrt{t}}(x))} \nonumber\\
    &\lesssim t^{\frac{6+n}{8+2n}} \|g\|_{L^{n+4}(Q_{\sqrt{t}}(x))}
    =t^{\frac{1}{2}}t^{\frac{1}{n+4}}\|g\|_{L^{n+4}(Q_{\sqrt{t}}(x))}\nonumber\\
    &\lesssim \frac{t^{\frac{1}{2}}}{1+|x|^\beta}.
\label{I1}
  \end{align}

For $I_2$, we estimate it as follows:
    \begin{align*}
      I_2&=\int_{\R^n\times(0, t)\setminus Q_{\sqrt{t}}(x)}|h(x-y, t-s)g(y, s)|dyds\\
      &\lesssim \sum_{m=0}^{\infty}\sum_{z\in 2^{-\frac{m}{2}}\sqrt{t}\Zn}\int_{2^{-m-1}t}^{2^{-m}t}\int_{B_{2^{-\frac{m}{2}}\sqrt{t}}(z)}t^{-\frac{n}{2}}e^{-c\frac{|x-y|}{\sqrt{t}}}|g(y,s)|dyds\\
      &= \sum_{m=0}^{\infty}\sum_{z\in 2^{-\frac{m}{2}}\sqrt{t}\Zn}\int_{Q_{2^{-\frac{m}{2}}\sqrt{t}}(z)}t^{-\frac{n}{2}}e^{-c\frac{|x-y|}{\sqrt{t}}}|g(y,s)|dyds\\
 &\lesssim\left(\sum_{m=0}^{\infty}\sum_{\substack{z\in \ttm \sqt\Zn\\|z-x|\leq \frac{\sqt|x|}{2}}}+\sum_{m=0}^{\infty}\sum_{\substack{z\in \ttm \sqt\Zn\\|z-x|\geq \frac{\sqt|x|}{2}}}\right) \int_{Q_{\ttm\sqt}(z)}t^{-\frac{n}{2}}e^{-c\frac{|x-y|}{\sqt}}|g(y, s)|dyds\\
 &:=I_{21}+I_{22}.
    \end{align*}
To estimate $I_{21}$, we compute,
\begin{align*}
  I_{21}
&\lesssim \sum_{m=0}^{\infty}\sum_{\substack{z\in 2^{-\frac{m}{2}}\sqrt{t}\Zn\\ |z-x|\le \frac{\sqrt{t}|x|}{2}}}e^{-c\frac{|x-z|}{\sqrt{t}}}\int_{Q_{2^{-\frac{m}{2}}\sqrt{t}}(z)}t^{-\frac{n}{2}}|g(y, s)|dyds\\
&\lesssim \sum_{m=0}^{\infty}
\left(\sup_{\substack{z\in 2^{-\frac{m}{2}}\sqrt{t}\Zn\\ |z-x|\le \frac{\sqrt{t}|x|}{2}}}\int_{Q_{2^{-\frac{m}{2}}\sqrt{t}}(z)}t^{-\frac{n}{2}}|g(y, s)|dyds\right)
\left(\sum_{\substack{z\in 2^{-\frac{m}{2}}\sqrt{t}\Zn\\ |z-x|\le \frac{\sqrt{t}|x|}{2}}}
e^{-c\frac{|z-x|}{\sqrt{t}}}\right),
\end{align*}
where we have used the estimate $|\sum_z a(z)b(z)| \leq \sup_z |a(z)|\sum_z |b(z)|$.
Note that
\begin{align*}
\sum_{\substack{z\in 2^{-\frac{m}{2}}\sqrt{t}\Zn\\ |z-x|\le \frac{\sqrt{t}|x|}{2}}}e^{-c\frac{|z-x|}{\sqrt{t}}}
&
\le \sum_{\substack{z\in 2^{-\frac{m}{2}}\sqrt{t}\Zn}}e^{-c\frac{|z|}{\sqrt{t}}}
=\sum_{\substack{\tilde{z}\in\Zn}}e^{-c|\tilde{z}|2^{-\frac{m}{2}}}
\int_{\mathbb{R}^n} e^{-c|\tilde{z}|2^{-\frac{m}{2}}}\,d^n\tilde{z}
\approx 2^{\frac{mn}{2}}
\end{align*}
while
\begin{align*}
&\sup_{\substack{z\in 2^{-\frac{m}{2}}\sqrt{t}\Zn\\ |z-x|\le \frac{\sqrt{t}|x|}{2}}}\int_{Q_{2^{-\frac{m}{2}}\sqrt{t}}(z)}t^{-\frac{n}{2}}|g(y, s)|dyds\\
\leq &t^{-\frac{n}{2}}
\sup_{\substack{z\in 2^{-\frac{m}{2}}\sqrt{t}\Zn\\ |z-x|\le \frac{\sqrt{t}|x|}{2}}}
\|1\|_{L^\frac{n+4}{n+3}(Q_{2^{-\frac{m}{2}}\sqrt{t}}(z))}
\|g\|_{L^{n+4}(Q_{2^{-\frac{m}{2}}\sqrt{t}}(z))}
\\
\le &t^{-\frac{1}{2}}2^{\frac{m(2-(n+2)(n+3))}{2(n+4)}}
\sup_{\substack{z\in 2^{-\frac{m}{2}}\sqrt{t}\Zn\\ |z-x|\le \frac{\sqrt{t}|x|}{2}}}
\left( 2^{-\frac{m}{2}}\sqrt{t} \right)^{\frac{2}{n+4}}\|g\|_{L^{n+4}(Q_{2^{-\frac{m}{2}}\sqrt{t}}(z))}\\
  \le & t^{\frac{1}{2}}2^{\frac{m(2-(n+2)(n+3))}{2(n+4)}} \sup_{\substack{z\in 2^{-\frac{m}{2}}\sqrt{t}\Zn\\ |z-x|\le \frac{\sqrt{t}|x|}{2}}}\frac{1}{1+|z|^\beta}
  \,\,\lesssim\,\,
 \frac{t^{\frac{1}{2}}2^{\frac{m(2-(n+2)(n+3))}{2(n+4)}}}{1+|x|^\beta}.
\end{align*}
Hence
\begin{align}
I_{21} 
& \le
\frac{t^{\frac{1}{2}}}{1+|x|^\beta}
\sum_{m=0}^\infty 
2^{\frac{m(2-(n+2)(n+3))}{2(n+4)}}2^{\frac{mn}{2}}
= 
\frac{t^{\frac{1}{2}}}{1+|x|^\beta}
\sum_{m=0}^\infty 2^{-\frac{m}{2}}
\lesssim
\frac{t^{\frac{1}{2}}}{1+|x|^\beta}.
\label{I21}
\end{align}
For $I_{22}$, we estimate it as
\begin{align*}
  I_{22}&\lesssim 
\sum_{m=0}^{\infty}\sum_{\substack{z\in 2^{-\frac{m}{2}}\sqrt{t}\Zn\\ |z-x|\ge \frac{\sqrt{t}|x|}{2}}}e^{-\frac{c}{4}|x|}\int_{Q_{2^{-\frac{m}{2}}\sqrt{t}}(z)}t^{-\frac{n}{2}}e^{-\frac{c}{2}\frac{|x-y|}{\sqrt{t}}}|g(y, s)|dyds\\
& \lesssim
e^{-\frac{c}{4}|x|}\sum_{m=0}^{\infty}
\sum_{\substack{z\in 2^{-\frac{m}{2}}\sqrt{t}\Zn\\ |z-x|\ge \frac{\sqrt{t}|x|}{2}}}\int_{Q_{2^{-\frac{m}{2}}\sqrt{t}}(z)}t^{-\frac{n}{2}}e^{-\frac{c}{2}\frac{|x-y|}{\sqrt{t}}}|g(y, s)|dyds\\
&\lesssim 
e^{-\frac{c}{4}|x|}\sum_{m=0}^{\infty}
\left(\sup_{\substack{z\in 2^{-\frac{m}{2}}\sqrt{t}\Zn\\ |z-x|\ge \frac{\sqrt{t}|x|}{2}}}\int_{Q_{2^{-\frac{m}{2}}\sqrt{t}}(z)}t^{-\frac{n}{2}}|g(y, s)|dyds\right)
\left(\sum_{\substack{z\in 2^{-\frac{m}{2}}\sqrt{t}\Zn\\ |z-x|\ge \frac{\sqrt{t}|x|}{2}}}
e^{-\frac{c}{2}\frac{|z-x|}{\sqrt{t}}}\right).
\end{align*}
Then similar to the computation for $I_{21}$, we arrive at
\begin{align}
  I_{22}
  \lesssim e^{-\frac{c}{4}|x|} t^{\frac{1}{2}}
\sum_{m=0}^{\infty} 
\left(\sup_{z\in 2^{-\frac{m}{2}}\sqrt{t}\Zn}\frac{1}{1+|z|^\beta}\right)
2^{-\frac{m}{2}}
 \lesssim e^{-\frac{c}{4}|x|}t^{\frac{1}{2}}
\lesssim \frac{t^{\frac{1}{2}}}{1+|x|^\beta}.  
\label{I22}
\end{align}
Combining \eqref{I1}, \eqref{I21}, and \eqref{I22}, we obtain
\begin{equation}
\sup_{0<t<T}\|(1+|x|^\beta) Sg(x, t)\|_{L^\infty(\R^n)}\lesssim t^{\frac{1}{2}}
\lesssim 1.
\label{0mcf}
\end{equation}

We re-state the estimate $I_2$ here for future usage:
\begin{align}
I_2
&=\int_{\R^n\times(0, t)\setminus Q_{\sqrt{t}}(x)}|h(x-y, t-s)g(y, s)|dyds
\nonumber\\
&\lesssim
\int_{\R^n\times(0, t)\setminus Q_{\sqrt{t}}(x)}
t^{-\frac{n}{2}}e^{-c\frac{|x-y|}{\sqt}}|g(y, s)|dyds
\lesssim
\frac{t^{\frac{1}{2}}}{1+|x|^\beta}.
\label{I2}
\end{align}
{\bf Estimate for $\nabla Sg$}. 
    \begin{align*}
      \left|\nabla Sg(x,t)\right|&=\left|\int_{0}^{t}\int_{\R^n}\nabla h(x-y, t-s)g(y, s)dyds\right|\\
      &\leq \int_{Q_{\sqt}(x)}+\int_{\R^n\times(0, t
      )\setminus Q_{\sqt}(x)}|\nabla h(x-y, t-s)g(y, s)|dyds\\
      &:=J_1+J_2.
    \end{align*}
    For $J_1$, by H\"{o}lder inequality, using the heat kernel estimate \eqref{eqn:LpGradientHeatKernel} with $p=\frac{n+4}{n+3} < \frac{n+2}{n+1}$, we can derive
    \begin{align}
      J_1&\leq \left\|\nabla h\right\|_{L^{\frac{n+4}{n+3}}(Q'_{\sqt}(0))}\left\|g\right\|_{L^{n+4}(Q_{\sqt}(x))}
      \leq \left\|\nabla h\right\|_{L^{\frac{n+4}{n+3}}(\R^n\times(0, t/2)}\left\|g\right\|_{L^{n+4}(Q_{\sqt}(x))}\nonumber\\
      &\lesssim \sqrt{t}^{\frac{2}{n+4}}\left\|g\right\|_{L^{n+4}(Q_{\sqt}(x))}
      \,\,\lesssim\,\,\XB.
\label{J1}
    \end{align}
For $J_2$, we can exactly follow the derivation of \eqref{I2}.
The only change is the appearance of $t^{-\frac{1}{2}}$ due to the 
point-wise estimate
of $\nabla h$ in \eqref{eqn:heatkernelexpdecay}.
    \begin{align}
    J_2  &=\int_{\R^n\times(0, t)\setminus Q_{\sqt}(x)}|\nabla h(x-y, t-s) g(y, s)|dyds\nonumber\\
      &\lesssim 
\int_{\R^n\times(0, t)\setminus Q_{\sqt}(x)}
t^{-\frac{1}{2}}t^{-\frac{n}{2}}e^{-c\frac{|x-y|}{\sqrt{t}}}|g(y, s)|dyds
\lesssim \frac{1}{1+|x|^\beta}\label{J2}.
    \end{align}
Combining \eqref{J1} and \eqref{J2}, we have
   \begin{equation}
   	\sup_{0<t<1}\sup_{x\in \R^n}(1+|x|^\beta)\left|\nabla Sg(x, t)\right|
\lesssim 1.
   	\label{eqn:Duhamelat1}
   \end{equation}
{\bf Estimate for $\nabla^2 Sg$}. For this, we need to show 
  \begin{equation}
    \sup_{0<R^2<1}R^{\frac{2}{n+4}}\left\|\nabla^2 Sg\right\|_{L^{n+4}(Q_R(x))}\lesssim\frac{1}{1+|x|^\beta}. 
    \label{eqn:5.12}
  \end{equation}
For this purpose, we compute
\begin{align*}
  &R^{\frac{2}{n+4}}\left\|\nabla^2 Sg(z, t)\right\|_{L^{n+4}(Q_R(x))}\\ =&R^{\frac{2}{n+4}}\left\|\int_{0}^{t}\int_{\R^n}\nabla^2 h(z-y,t-s)g(y, s)dyds\right\|_{L^{n+4}(Q_R(x))}\\
  =&R^{\frac{2}{n+4}}\Bigg\|\int_{R^n\times(0, t)\setminus B_{2R}(x)\times(R^2/4, R^2)}
  +\int_{B_{2R}(x)\times(R^2/4, R^2)}\nabla^2 h(z-y, t-s)g(y, s)dyds\Bigg\|_{L^{n+4}(Q_R(x))}\\
  =&R^{\frac{2}{n+4}}\Bigg\|\int_{R^n\times(0, t)\setminus B_{2R}(x)\times(R^2/4, R^2)} \nabla^2 h(z-y, t-s)g(y, s)dyds\Bigg\|_{L^     {n+4}(Q_R(x))}\\
& + R^{\frac{2}{n+4}}
  \Bigg\|\int_{B_{2R}(x)\times(R^2/4, R^2)}\nabla^2 h(z-y, t-s)g(y, s)dyds\Bigg\|_{L^{n+4}(Q_R(x))}\\
  :=& K_1+K_2.
\end{align*}

For $K_1$, we have
\begin{align}
  K_1&=R^{\frac{2}{n+4}}\left\|\int_{R^n\times(0, t)\setminus B_{2R}(x)\times(R^2/4, R^2)}\nabla^2 h(z-y, t-s)g(y, s)dyds\right\|_{L^{n+4}(Q_R(x))}\nonumber\\
  &\lesssim R^{\frac{2}{n+4}}\left\|\frac{t^{-\frac{1}{2}}}{1+|z|^\beta}\right\|_{L^{n+4}(Q_R(x))}
  \,\,\lesssim\,\, \frac{1}{1+|x|^\beta}, 
\label{K1}
\end{align}
where we have used again the estimate \eqref{I2}
for $I_2$ but with $h$ replaced by $\nabla^2h$. The $t^{-\frac{1}{2}}$ factor
is due to the pointwise estimate for $\nabla^2 h$ from 
\eqref{eqn:heatkernelexpdecay}. Note also $\frac{R^2}{2}< t < R^2$.

For $K_2$, 
\begin{align}
  K_2&:=R^{\frac{2}{n+4}}\left\|\int_{B_{2R}(x)\times(R^2/4, R^2)}\nabla^2 h(z-y, t-s)g(y, s)dyds\right\|_{L^{n+4}(Q_R(x))}\nonumber\\
   &\lesssim R^{\frac{2}{n+4}}\left\|\chi_{B_{2R}(x)\times(R^2/4, R^2)}g(z, t)\right\|_{L^{n+4}(\R^n\times\R_+)}\nonumber\\
  &\lesssim R^{\frac{2}{n+4}}\|g\|_{L^{n+4}(B_{2R}(x)\times(R^2/4, R^2))}
  \,\,\lesssim\,\,\XB.
\label{K2}
\end{align}
where the second inequality is due to \eqref{eqn:singintheat}.

Hence \eqref{eqn:5.12} holds upon combining \eqref{K1} and \eqref{K2}.
      \end{proof}

\section{Proof of Lemma \ref{lemma:keylemma2}}\label{appe:keylemma2}
The strategy here is very similar to Lemma \ref{lemma:keylemma1}. 
The main difference
is the usage of the estimates of the biharmonic kernel $b$ and also the fact 
that we need to deal with $g_l$ for $l=0,1,2$.
For $0\le k\le 3$ and $t>0$, we have from \eqref{eqn:KerDecay2} that,
  \begin{equation}
    \left\|\nabla^k b\right\|_{L^p(\R^n\times(0, t)}\leq Ct^{\frac{(n+4)-p(n+k)}{4p}}, \qquad 1\leq p< \frac{n+4}{n+k}, 
        \label{biharm-basic-est}
  \end{equation}
while for $k=4$, the following comes from the theory of singular integrals 
\cite{Stein},
\begin{equation}
  \left\|\int_{0}^{t}\int_{\R^n}\nabla^4 b(z-y, t-s)g(y,s)\,dyds\right\|_{L^p(\R^n\times\R^+)}\lesssim\left\|g\right\|_{L^p(\R^n\times\R^+)}.
  \label{singIntWF}
\end{equation}
Furthermore, from the scaling property of the kernel, the following pointwise 
estimates hold,
\begin{equation}
  \sum_{k=0}^{4}\left|(\fqt\nabla)^k b(z, s)\right|\lesssim t^{-\frac{n}{4}}\exp\left( -c\frac{|z|}{\fqt} \right), \quad \forall (y, s)\in \R^n\times(0, t)\setminus Q'_{\fqt}(0).
  \label{PtwiseBiharm}
\end{equation}
where we recall the notation,
$Q_R(x) = B_R(x)\times(\frac{R^4}{2},R^4)$ and 
$Q'_R(x) = B_R(x)\times(0,\frac{R^4}{2})$. 

\begin{proof}[Proof of Lemma \ref{lemma:keylemma2}] The proof follows a similar 
paradigm as in the previous section. 
It suffices to show that there exists a $C>0$ such that if $\sum_{l=0}^{2}\left\|g_l\right\|_{Y_{l, T}^\beta}\leq 1$, then 
$$\sum_{l=0}^{2}\sup_{0<t<T}\|(1+|x|^\beta)\nabla^l Sg_l(x , t)\|_{L^\infty(\R^n)}+\left\|\nabla^l Sg_l\right\|_{X_T^\beta} \leq C.
$$
Without loss of generality, we fix $T=1$. 
Note also 
$Q_{\fqt}(x):=B_{\fqt}(x)\times(t/2, t)$ and
$Q'_{\fqt}(x):=B_{\fqt}(x)\times(0, t/2)$.
Now we estimate the relevant quantities.

{\bf Estimate for $Sg_l$ ($0 \leq l < 2$).}
We compute,
  \begin{align*}
\left|\nabla^l S g_l(x, t)\right|
    &=\left|\int_{0}^{t}\int_{\R^n}\nabla^l b(x-y, t-s)g_l(y, s)dyds\right|\\
    &\leq \left(\int_{Q'_{\fqt}(0)}+\int_{\R^n\times(0, t)\setminus Q_{\fqt}(x)} \right) |\nabla^l b(x-y, t-s)g_l(y, s)|dyds\\
    &=I_1+I_2. 
  \end{align*}

For $I_1$, by the H\"{o}lder inequality, using the kernel estimate \eqref{biharm-basic-est} with $p=\frac{n+6}{n+3+l}< \frac{n+4}{n+l}$, we arrive at
  \begin{align}
    I_1\leq & \left\|\nabla^{l}b\right\|_{L^{\frac{n+6}{n+3+l}}(Q'_{\fqt}(0))}\left\|g_l\right\|_{L^{\frac{n+6}{3-l}}(Q_{\fqt}(x))}
    \leq \left\|\nabla^{l}b\right\|_{L^{\frac{n+6}{n+3+l}}(\R^n\times(0, t))}\left\|g_l\right\|_{L^{\frac{n+6}{3-l}}(Q_{\fqt}(x))}\nonumber\\
\lesssim&
t^{\frac{n+4-(n+l)p}{4p}}
t^{-\frac{1}{4}\left(\frac{6-2l}{n+6}\right)}
\left(t^{\frac{1}{4}\frac{6-2l}{n+6}}
\left\|g_{{l}}\right\|_{L^{\frac{n+6}{3-l}}(Q_{\fqt}(x))}\right)
    \lesssim \frac{t^{\frac{1}{4}}}{1+|x|^\beta}.
\label{I1w}
  \end{align}

For $I_2$, we make use of \eqref{PtwiseBiharm} and compute
\begin{align*}
  I_2&\lesssim\int_{\R^n\times(0, t)\setminus Q_{\fqt}(x)}\left|\nabla^lb(x-y, t-s)g_l(y, s)\right|dyds\\
  & \lesssim\sum_{m=0}^{\infty}\sum_{z\in \tfm\fqt\Zn}\int_{Q_{\tfm\fqt}(z)}t^{-\frac{n+l}{4}}e^{-c\frac{|x-y|}{\fqt}}|g_l(y, s)|dyds\\
  &\lesssim\left( \sum_{m=0}^{\infty}\sum_{\substack{z\in \tfm\fqt\Zn\\|z-x|\leq \frac{\fqt|x|}{2}}}+\sum_{m=0}^{\infty}\sum_{\substack{z\in\tfm\fqt\Zn\\|z-x|\geq\frac{\fqt|x|}{2}}} \right) \int_{Q_{\tfm\fqt}(x)}t^{-\frac{n+l}{4}}e^{-c\frac{|x-y|}{\fqt}}|g_l(y,s)|dyds\\
  &:=I_{21}+I_{22}. 
\end{align*}
Again, similar to the previous section, we have
\begin{align}
I_{21} 
= & \sum_{m=0}^{\infty}\sum_{\substack{z\in\tfm\fqt\Zn\\|z-x|\leq \frac{\fqt|x|}{2}}}\int_{Q_{\tfm\fqt}(z)}t^{-\frac{n+l}{4}}e^{-c\frac{|x-y|}{\fqt}}|g_l(y,s)|dyds\nonumber\\
\leq & t^{-\frac{n+l}{4}} \sum_{m=0}^{\infty}
\left(\sum_{\substack{z\in\tfm\fqt\Zn\\|z-x|\leq \frac{\fqt|x|}{2}}}
e^{-c\frac{|x-z|}{\fqt}}\right)
\left(\sup_{\substack{z\in\tfm\fqt\Zn\\|z-x|\leq \frac{\fqt|x|}{2}}}\int_{Q_{\tfm\fqt}(z)}|g_l(y,s)|dyds\right)\nonumber\\
\lesssim&
t^{-\frac{n+l}{4}} \sum_{m=0}^{\infty} 2^{\frac{mn}{4}}
\left(\sup_{\substack{z\in\tfm\fqt\Zn\\|z-x|\leq \frac{\fqt|x|}{2}}}
\|1\|_{L^{\frac{n+6}{n+3+l}}\left(Q_{\tfm\fqt}(z)\right)}
\|g_l\|_{L^{\frac{n+6}{3-l}}\left(Q_{\tfm\fqt}(z)\right)}
\right)\nonumber\\
\lesssim & 
t^{\frac{1}{4}}\sum_{m=0}^\infty \frac{2^{-\frac{m(1+l)}{4}}}{1+|x|^\beta}
\lesssim 
\frac{t^{\frac{1}{4}}}{1+|x|^\beta}
\label{I21w}
\end{align}
while for $I_{22}$,
\begin{align}
I_{22}
= & \sum_{m=0}^{\infty}\sum_{\substack{z\in\tfm\fqt\Zn\\|z-x|> \frac{\fqt|x|}{2}}}\int_{Q_{\tfm\fqt}(z)}t^{-\frac{n+l}{4}}e^{-c\frac{|x-y|}{\fqt}}|g_l(y,s)|dyds\nonumber\\
\leq & t^{-\frac{n+l}{4}} e^{-\frac{c}{4}|x|}\sum_{m=0}^{\infty}
\left(\sum_{\substack{z\in\tfm\fqt\Zn\\|z-x|>\frac{\fqt|x|}{2}}}
e^{-c\frac{|x-z|}{\fqt}}\right)
\left(\sup_{\substack{z\in\tfm\fqt\Zn\\|z-x|> \frac{\fqt|x|}{2}}}\int_{Q_{\tfm\fqt}(z)}|g_l(y,s)|dyds\right)\nonumber\\
\lesssim&
t^{-\frac{n+l}{4}} e^{-\frac{c}{4}|x|}\sum_{m=0}^{\infty} 2^{\frac{mn}{4}}
\left(\sup_{z\in\tfm\fqt\Zn}
\|1\|_{L^{\frac{n+6}{n+3+l}}\left(Q_{\tfm\fqt}(z)\right)}
\|g_l\|_{L^{\frac{n+6}{3-l}}\left(Q_{\tfm\fqt}(z)\right)}
\right)\nonumber\\
\lesssim &
t^{\frac{1}{4}}e^{-\frac{c}{4}|x|}
\sum_{m=0}^\infty 2^{-\frac{m(1+l)}{4}}
\sup_{z\in\tfm\fqt\Zn}
\frac{1}{1+|z|^\beta}
\lesssim
\frac{t^{\frac{1}{4}}}{1+|x|^\beta}.
\label{I22w}
\end{align}

Combining \eqref{I1w}, \eqref{I21w} and \eqref{I22w} leads to
$$
\sup_{0<t<1}\sup_{x\in\R^n} (1+|x|^\beta)\sum_{l=0}^{2}|\nabla^lSg_l(x, t)|
\lesssim
t^{\frac{1}{4}}\sum_{l=0}^{2}\left\|g_l\right\|_{Y_{l,1}^\beta}
\lesssim
\sum_{l=0}^{2}\left\|g_l\right\|_{Y_{l,1}^\beta}.
$$

{\bf Estimate for $\nabla Sg_l$ ($0 \leq l < 2$).}
The same computation leads to
$$
\sup_{0<t<1}\sup_{x\in\R^n} (1+|x|^\beta)
\sum_{l=0}^{2}|\nabla^l\nabla Sg_l(x, t)|
\lesssim
\sum_{l=0}^{2}\left\|g_l\right\|_{Y_{l,1}^\beta}.
$$
This is essentially the same as going from \eqref{0mcf} to 
\eqref{eqn:Duhamelat1}. Hence we just outline the key computation.
  \begin{align*}
    \left|\nabla \nabla^l S g_l(x, t)\right|
    &=\left|\int_{0}^{t}\int_{\R^n}\nabla \nabla^l b(x-y, t-s)g_l(y, s)dyds\right|\\
    &\leq \left(\int_{Q'_{\fqt}(0)}+\int_{\R^n\times(0, t)\setminus Q_{\fqt}(x)} \right) |\nabla \nabla^l b(x-y, t-s)g_l(y, s)|dyds\\
    &:=J_1+J_2.
  \end{align*}

For $J_1$, by the H\"{o}lder inequality, using the kernel estimate 
\eqref{biharm-basic-est} with $p=\frac{n+6}{n+3+l} < \frac{n+4}{n+l+1}$, 
$l=0,1,2$ so that we can derive
  \begin{align*}
    J_1 \leq & \left\|\nabla^{l+1}b\right\|_{L^{\frac{n+6}{n+3+l}}(Q'_{\fqt}(0))}\left\|g_l\right\|_{L^{\frac{n+6}{3-l}}(Q_{\fqt}(x))}
    \leq  \left\|\nabla^{l+1}b\right\|_{L^{\frac{n+6}{n+3+l}}(\R^n\times(0, t))}\left\|g_l\right\|_{L^{\frac{n+6}{3-l}}(Q_{\fqt}(x))}\\
    \lesssim& \fqt^{\frac{6-2l}{n+6}}\left\|g_{{l}}\right\|_{L^{\frac{n+6}{3-l}}(Q_{\fqt}(x))}
    \lesssim \XB.
  \end{align*}
For $J_2$, the computation is similar. The extra
factor $t^{-\frac{1}{4}}$ coming from $\nabla^{l+1} b$ is absorbed
by the $t^{\frac{1}{4}}$ in \eqref{I1w}, \eqref{I21w}, and \eqref{I22w}.

{\bf Estimate for $\nabla^2 Sg_l$ ($0 \leq l < 2$).}
For this, we need to show 
\begin{equation}
  \sup_{0<R^4<1}R^{\frac{2}{n+6}}\left\|\nabla^{2+l}Sg_l(z,t)\right\|_{L^{n+6}(Q_R(x))}\lesssim\frac{1}{1+|x|^\beta}, \quad l=0,1,2.
  \label{eqn:hiorderWF}
\end{equation}
We first compute,
\begin{align*}
& \|\nabla^{2+l}Sg_l\|_{L^{n+6}(Q_R(x))}\\
= & 
\left\|\int_{\R^n\times(0, t)\setminus B_{2R(x)}\times (R^4/4, R^4)}
+ \int_{B_{2R(x)}\times (R^4/4, R^4)}
\nabla^{2+l}b(z-y, t-s)g_l(y, s)dyds\right\|_{L^{n+6}(Q_R(x))}\\
\leq & 
\left\|\int_{\R^n\times(0, t)\setminus B_{2R(x)}\times (R^4/4, R^4)}
\nabla^{2+l}b(z-y, t-s)g_l(y, s)|dyds\right\|_{L^{n+6}(Q_R(x))}\\
& + \left\|\int_{B_{2R(x)}\times (R^4/4, R^4)}
\nabla^{2+l}b(z-y, t-s)g_l(y, s)dyds\right\|_{L^{n+6}(Q_R(x))}\\
:= & K_1 + K_2.
\end{align*}

For $K_1$, using the same arguments as the $K_1$ in the previous section,
we get the pointwise bound
$$
\int_{\R^n\times(0, t)\setminus B_{2R(x)}\times (R^4/4, R^4)}
|\nabla^{2+l}b(z-y, t-s)g(y, s)|dyds
\lesssim\frac{t^{-\frac14}}{1+|z|^\beta}
$$
so that
\begin{align}
R^{\frac{2}{n+6}}\left\|\frac{t^{-\frac14}}{1+|z|^\beta}\right\|_{L^{n+6}(Q_R(x))}
\lesssim R^{\frac{2}{n+6}}
\frac{t^{-\frac{1}{4}}}{1+|x|^\beta}
\left(R^nR^4\right)^{\frac{1}{n+6}}
\approx
\frac{1}{1+|x|^\beta}
\label{2wf-1}
\end{align}
where we have used the fact that $\frac{R^4}{2}<t<R^4$.

For $K_2$, we can focus on the $L^{n+6}$-estimate for $\nabla^{2+l}Sg_l$ with 
$g_l$ supported in $Q_{2R}(x).$ First we recall the Young inequality:
\begin{equation*}
  \left\|f*g\right\|_{L^m(\R^n\times\R^+)}\leq C\left\|f\right\|_{L^p(\R^n\times\R^+)}\left\|g\right\|_{L^q(\R^n\times\R^+)},
\end{equation*}
where $0\leq p, q, m\leq \infty$, and $p^{-1}+q^{-1}=1+m^{-1}$. Applying the inequality with 
$m=n+6$, $p=\frac{n+6}{n+4} < \frac{n+4}{n+2}$, $q=\frac{n+6}{3}$, 
respectively 
$m=n+6$, $p=\frac{n+6}{n+5} < \frac{n+4}{n+3}$, $q=\frac{n+6}{2}$, 
we get
\begin{align*}
  \left\|\nabla^2 S g_0\right\|_{L^{n+6}(\R^n\times (0, 1))}&\lesssim\left\|g_0\right\|_{L^{\frac{n+6}{3}}(\R^n\times\R^+)}=\left\|g_0\right\|_{L^{\frac{n+6}{3}}(B_{2R}(x)\times(R^4/4, R^4))}, \\
\text{and}\,\,\,
\left\|\nabla^3 S g_1\right\|_{L^{n+6}(\R^n\times(0, 1))}&\lesssim\left\|g_1\right\|_{L^{\frac{n+6}{2}}(\R^n\times\R^+)}=\left\|g_1\right\|_{L^{\frac{n+6}{2}}(B_{2R}(x)\times(R^4/4, R^4))}.
\end{align*}
Hence
\begin{align}
R^{\frac{2}{n+6}}\left\|\nabla^2 S g_0\right\|_{L^{n+6}(\R^n\times (0, 1))},
\,\,\,
R^{\frac{2}{n+6}}
\left\|\nabla\nabla^2 S g_1\right\|_{L^{n+6}(\R^n\times (0, 1))}
\lesssim \XB.
\label{2wf-2}
\end{align}
For the $L^{n+6}$ norm of $\nabla^4 Sg_2$, by the singular integral estimate
\eqref{singIntWF} with $p=n+6$, we have that
\begin{align}
  &R^{\frac{2}{n+6}}\left\|\int_{B_{2R}(x)\times(R^4/4, R^4)}\nabla^4 b(z-y, t-s)g_2(y, s)dyds\right\|_{L^{n+6}(Q_R(x))}\\
  \lesssim &R^{\frac{2}{n+6}}\left\|g_2\right\|_{L^{n+6}(B_{2R}(x)\times(R^4/4, R^4))}
 \lesssim\frac{1}{1+|x|^\beta}.
\label{2wf-3}
\end{align}
Combining \eqref{2wf-1}, \eqref{2wf-2}, and \eqref{2wf-3} gives
\eqref{eqn:hiorderWF}, completing the proof. 
\end{proof}

\noindent
\textbf{Acknowledgment.} The authors thank Changyou Wang for useful discussion.


\end{document}